\documentclass[11pt,oneside]{article}

\usepackage{pzorin-symbols}
\usepackage{pzorin-fonts}
\usepackage{hyperref}
\hypersetup{
  pdfstartview=FitH,
  pdfauthor={Pavel Zorin-Kranich},
  pdftitle={Martingale inequalities},
  pdfborder={0 0 0},
  bookmarks=true,
  unicode
}

\usepackage[paperheight=279mm,paperwidth=16cm,textheight=26cm,textwidth=14cm]{geometry}
\usepackage[english]{babel}

\usepackage[style=alphabetic]{biblatex}
\newcommand{\bmin}{\wedge}
\newcommand{\bmax}{\vee}
\newcommand{\Ito}{It\^o}

\newcommand{\pred}{\mathrm{pred}}
\newcommand{\bv}{\mathrm{bv}}
\DeclarePairedDelimiterXPP\EE[1]{\E}{\lparen}{\rparen}{}{\renewcommand\given{\SetSymbol[\delimsize]}#1} 
\def\dist#1#2{\abs{#1-#2}}
\def\stoptime#1#2{\tau^{(#1)}_{#2}}
\newcommand{\startat}[2]{\prescript{#1}{}\!#2}
\def\P{\mathbb{P}}
\def\Pimax{\Pi^*}
\newcommand\cadlag{c\`{a}dl\`{a}g}

\newcounter{vorl}
\def\vorllabel#1{[\arabic{vorl}:\, #1]}
\def\neuevorl#1{\clearpage \hfill \lastvorl%
\hrule \stepcounter{vorl}\global\def\lastvorl{\vorllabel{#1}}%
\hfill \lastvorl \\}
\def\lastvorl{}
\def\neuevorl#1{}

\begin{document}
\setcounter{section}{-1}

\title{Martingale inequalities}
\author{Pavel Zorin-Kranich}
\date{University of Bonn, Winter term 2021/22}
\maketitle

These lecture notes cover a few techniques for proving $L^{p}$ estimates for martingales and the most basic applications to \Ito{} integration and rough paths.
They have been created for a course that consisted of 13 lectures \'a 90 minutes.
Due to leaving academia, I didn't have time to polish these notes much, but you are welcome to, because:

This work is licensed under the Creative Commons Attribution 4.0 International License.
To view a copy of this license, visit
\begin{center}
\url{https://creativecommons.org/licenses/by/4.0/}
\end{center}
or send a letter to Creative Commons, PO Box 1866, Mountain View, CA 94042, USA.

\paragraph*{Literature}
There are two books whose titles include our topic \cite{MR0448538,MR1224450}.
They can be useful for some of the basics, but are overall outdated.

For vector-valued martingales, modern references are \cite{MR3617205,MR3617459}.

The only book about sharp constants in martingale inequalties is \cite{MR2964297}.

An extensive treatment of rough paths can be found in \cite{MR4174393}.

One stochastic analysis text that pays attention to inequalities is \cite{MR4226142}.

\tableofcontents

\clearpage

\neuevorl{2021-10-12}
\section{Review of martingale basics}
Random variables will be usually defined on a filtered probability space $(\Omega,(\calF_{n})_{n\in\N},\mu)$.
We denote by $\calF_{\infty}$ the $\sigma$-algebra generated by $\cup_{n\in\N} \calF_{n}$.
The following examples are useful to keep in mind.
\begin{example}[Dyadic filtration]
\label{ex:dyadic-filtration}
$\Omega = [0,1]$, $\mu$ Lebesgue measure, $\calF_{n}$ is the $\sigma$-algebra generated by the dyadic intervals of length $2^{-n}$, that is, intervals of the form $[2^{-n}k,2^{-n}(k+1)]$ with $k\in\Z$.

A higher dimensional version involves dyadic cubes.
One can also construct similar filtrations on more general manifolds, or even metric measure spaces.
\end{example}

\begin{example}[Atomic filtrations]
$\Omega = [0,1]$, $\mu$ Lebesgue measure, $\calF_{n}$ is a $\sigma$-algebra generated by finitely many intervals.

We recall that an \emph{atom} for a measure $\mu$ on a $\sigma$-algebra $\calF$ is a set $A\in\calF$ such that $\mu(A)>0$ and, for every $A' \in \calF$ with $A' \subseteq A$, we have $\mu(A') \in \Set{0,\mu(A)}$.
In an atomic $\sigma$-algebra, every measurable set is a finite union of atoms.

One can view martingale analysis as analysis of atomic filtrations, if we do not allow any constants to depend on the filtration.
All results that we are interested in can be transferred from atomic filtrations to general filtrations.
However, it is technically convenient to always use general filtrations, since they can appear in applications.
\end{example}

\begin{example}
$\Omega = [0,1]^{\N}$, $\mu$ the product of Lebesgue measures,
\begin{equation}
\label{eq:1}
\calF_{n} = \Set{ B \times [0,1]^{\Set{n,n+1,\dotsc}} \given B \subseteq [0,1]^{\Set{0,\dotsc,n-1}} \text{ Borel}}.
\end{equation}
This filtration appears in the analysis of independent random variables, and gives a good idea of how a general filtration looks like.
\end{example}

An \emph{adapted process} is a sequence of functions $(f_{n})$ such that, for every $n\in\N$, the function $f_{n}$ is $\calF_{n}$-measurable.

An adapted process $f$ is called \emph{predictable} if, for every $n>0$, the function $f_{n}$ is $\calF_{n-1}$-measurable.

For nested $\sigma$-algebras $\calF' \subseteq \calF$ on $\Omega$, the \emph{conditional expectation} is the orthogonal projection
\[
\E(\cdot|\calF') : L^{2}(\Omega,\calF,\mu) \to L^{2}(\Omega,\calF',\mu).
\]
The conditional expectation has the following properties.
Here and later, all identities and inequalities are meant to hold almost surely, unless mentioned otherwise.
\begin{enumerate}
\item\label{it:ce-one} $\E(\one|\calF')=\one$.
\item\label{it:ce-Lp-contr} For every $p\in [1,\infty]$, $\E$ extends to a contraction $L^{p}(\calF) \to L^{p}(\calF')$
\item\label{it:ce-int} $\int \E(f|\calF') = \int f$ for every $f\in L^{1}(\calF)$.
\item\label{it:ce-pos} Positivity: $f\geq 0 \implies \E(f|\calF')\geq 0$.
\item\label{it:ce-prod} Assume that $f\in L^{1}(\calF)$, $g\in L^{0}(\calF')$, and either $fg\in L^{1}(\calF)$ or $f\geq 0, \E(f|\calF')g\in L^{1}(\calF')$. Then
\[
\E(fg|\calF') = \E(f|\calF')g \text{ in } L^{1}(\calF').
\]
\end{enumerate}

Suppose that $(X,\calF,\mu)$ is a regular measure space and $\calF' \subseteq \calF$ is a sub-$\sigma$-algebra.
Then there exists an essentially unique measurable map $(X,\calF') \to \calM(X)$, $y\mapsto\mu_{y}$ such that, for every $f \in L^{1}(\calF)$, for $\mu$-a.e.\ $y\in X$, we have $f\in L^{1}(X,\mu_{y})$, and $\int f \dif\mu_{y} = \E(f|\calF')(y)$.
This map is called a \emph{measure disintegration}.
A measure disintegration satisfies $\mu_{x}=\mu_{y}$ for $\mu$-a.e.\ $y$ and $\mu_{y}$-a.e.\ $x$.

\begin{example}
If $\calF'$ is an atomic filtration, then we can choose a collection of disjoint atoms $\calA \subseteq \calF'$ with $\mu(\cup_{A\in\calA} A)=1$.
Then, for each $x\in A \in \calA$, we can set $\dif\mu_{x} := \mu(A)^{-1} \one_{A} \dif\mu$, and
\[
\E(f | \calF')(x) = \mu(A)^{-1} \int_{A} f(x') \dif\mu(x').
\]
\end{example}

\begin{example}
Let $\Omega = [0,1]^{2}$ with the Lebesgue measure, $\calF$ the Borel $\sigma$-algebra, and $\calF'$ the Borel $\sigma$-algebra in the first variable, like in \eqref{eq:1}.
Then, we can choose $\mu_{(x,y)}$ to be the Lebesgue measure on $\Set{x} \times [0,1]$, and
\[
\E(f | \calF')(x,y) = \int_{0}^{1} f(x,y') \dif y'.
\]
\end{example}

A \emph{martingale} is an adapted process with values in $\bbC$ (later also in a Banach space) such that, for every $m \leq n$, the function $f_{n}$ is integrable, and we have
\[
f_{m} = \E(f_{n} | \calF_{m}).
\]
\begin{example}
Sums of independent random variables
\end{example}

\begin{example}[Dyadic martingale]
Let $f$ be an integrable function on $[0,1]$ with the Lebesgue measure and $f_{n} = \E(f|\calF_{n})$ (the same definition works for any filtration).

This construction can often be used to transfer results from martingales to a real analysis setting.
The main difference between the analysis of dyadic martingales and general martingales is that the Calder\'on--Zygmund decomposition does not work for general martingales.
A substitute that does work for general martingales is the Gundy decomposition.
\end{example}

The main reason why this construction does not produce all possible martingales is that it may happen that $\lim_{n\to\infty}\norm{f_{n}}=\infty$, in which case there might be no function $f$ with $f_{n} = \E(f|\calF_{n})$.
A well-known example involves a doubling betting strategy.

\begin{example}
Any integrable process can be written as the sum of a process with predictable jumps and a martingale.
\end{example}

A \emph{stopping time} is a function $\tau : \Omega \to \bar{\N} = \N \cup \Set{\infty}$ such that, for every $n\in\N$, we have $\Set{\tau \leq n} \in \calF_{n}$.

\begin{example}
Any constant function is a stopping time.
\end{example}

\begin{example}[Hitting time]
If $f$ is an adapted process with values in a metric space $X$ and $B \subseteq X$ is a Borel set, then
\[
\tau := \inf \Set{t \given f_{t} \in B}
\]
is a stopping time, called the \emph{first hitting time} of $B$.
\end{example}

If $\sigma,\tau$ are stopping times, then $\sigma \bmin \tau$ and $\sigma \bmax \tau$ are also stopping times.

For a stopping time $\tau$ and a process $f$, the \emph{stopped process} $f^{\tau}$ is defined by
\[
f^{\tau}_{t} := f_{\tau \bmin t}.
\]
If $f$ is a martingale, then the stopped process $f^{\tau}$ is again a martingale.

If $\tau$ is a stopping time, the corresponding $\sigma$-algebra is defined by
\[
\calF_{\tau} := \Set{ A\in\calF_{\infty} \given (\forall n) A \cap \Set{\tau \leq n} \in \calF_{n}}.
\]
The function $\tau$ is $\calF_{\tau}$-measurable.
We abbreviate $\E_{\tau} f := \E(f | \calF_{\tau})$.

The \emph{optional sampling} theorem says that, for every discrete time martingale $f$, bounded stopping time $\tau$, and another stopping time $\sigma$, we have
\[
f_{\sigma \bmin \tau} = \E_{\sigma} f_{\tau}.
\]

\section{Maximal and square functions}

You are probably already familiar with the Lebesgue differentialtion theorem.
The best proof of that theorem uses the maximal operator to absorb error terms.
This idea is used in many places in analysis, and the study of martingales is one of these places.

\begin{definition}[Maximal operator]
For an adapted process $f$ with values in a normed space, we write
\[
Mf_{n} = (Mf)_{n} := \sup_{k\leq n} \abs{f_{k}}.
\]
\end{definition}
\begin{remark}
Since our martingales are indexed by a countable set, we can use a pointwise supremum here.
In continuous time, we would need a lattice supremum.
\end{remark}

A \emph{submartingale} is an adapted process with values in $\R_{\geq 0}$ such that, for every $m \leq n$, we have
\[
f_{m} \leq \E_{m}f_{n}.
\]
\begin{example}
If $(f_{n})$ is a martingale, then $(\abs{f_{n}})$ is a submartingale.
\end{example}
\begin{example}
If $f$ is any adapted process, then $Mf$ is a submartingale.
\end{example}

If $f$ is a submartingale, then, for any $p\in [1,\infty]$ and $m\leq n$, we have
\[
\norm{f_{m}}_{p} \leq \norm{f_{n}}_{p}.
\]
We define the $L^{p}$ norm of a (sub-)martingale by
\[
\norm{f}_{p} := \sup_{n} \norm{f_{n}}_{p}.
\]
With this definition, the maximal operator is clearly bounded on $L^{\infty}$.
The next result looks a lot like an $L^{1}\to L^{1,\infty}$ bound for $M$, but is in fact stronger and more convenient to use.

\begin{lemma}
\label{lem:Doob-weak}
Let $f$ be a submartingale with values in $\R_{\geq 0}$.
Then, for every $\lambda>0$ and $n\in\N$, we have
\[
\lambda \abs{\Set{Mf_{n} > \lambda}}
\leq
\int_{\Set{Mf_{n}>\lambda}} f_{n} \dif\mu.
\]
\end{lemma}
\begin{proof}
For a fixed $\lambda$, define the stopping time
\[
\tau := \inf \Set{k \given f_{k} > \lambda}.
\]
Then, $\Set{Mf_{n} > \lambda} = \Set{\tau \leq n}$.
Hence,
\begin{align*}
\lambda \abs{\Set{Mf_{n} > \lambda}}
&=
\lambda \sum_{k \leq n} \abs{\Set{\tau = k}}
\\ &\leq
\sum_{k \leq n} \int_{\Set{\tau = k}} f_{k} \dif\mu
\\ \text{(submartingale property)} &\leq
\sum_{k \leq n} \int_{\Set{\tau = k}} f_{n} \dif\mu
\\ &=
\int_{\Set{Mf_{n} > \lambda}} f_{n} \dif\mu.
\qedhere
\end{align*}
\end{proof}

\begin{remark}[$L^{1,\infty}$ norm]
By Chebychev's inequality, for any measurable function $g$, we have
\[
\sup_{\lambda>0} \lambda^{-1} \abs{ \Set{g>\lambda} }
\leq \norm{g}_{1}.
\]
The converse inequality is false with any constant, as shown by the example $g(x)=1/x$ on the measure space $\R_{>0}$ with the Lebesgue measure.
The left-hand side of the above inequality is the so-called $L^{1,\infty}$ seminorm of $g$.
As suggested by the notation, $L^{1,\infty}$ is part of a larger family of Lorentz spaces, which is in turn contained in Orlicz spaces, but we will not talk about such generalities.
\end{remark}

\begin{theorem}[Doob's maximal inequality]
\label{thm:Doob-Lp}
Let $f$ be a submartingale with values in $\R_{\geq 0}$.
Then, for every $p \in (1,\infty]$, we have
\[
\norm{Mf}_{p} \leq p' \norm{f}_{p}.
\]
Here and later, $p'$ denotes the Hölder conjugate: $1/p + 1/p' = 1$.
\end{theorem}
\begin{proof}
By the layer cake formula and Lemma~\ref{lem:Doob-weak}, we write
\begin{align*}
\norm{Mf_{n}}_{p}^{p}
&=
\int_{0}^{\infty} p \lambda^{p-1} \mu(Mf_{n}>\lambda) \dif \lambda
\\ &\leq
\int_{0}^{\infty} p \lambda^{p-2} \int_{\Set{Mf_{n}>\lambda}} f_{n} \dif\mu \dif \lambda
\\ &=
\int_{\Omega} \int_{0}^{Mf_{n}} p \lambda^{p-2} f_{n} \dif \lambda \dif\mu
\\ &=
\frac{p}{p-1} \int_{\Omega} (Mf_{n})^{p-1} f_{n} \dif\mu
\\ &\leq
\frac{p}{p-1} (\int_{\Omega} f_{n}^{p} \dif\mu)^{1/p} (\int_{\Omega} (Mf_{n})^{p} \dif\mu)^{1-1/p}.
\end{align*}
Since $f_{n} \in L^{p}$ implies $f_{k} \in L^{p}$ for all $k\leq n$ by $L^{p}$ contractivity of conditional expectations, we have $Mf_{n} \in L^{p}$.
Hence, we can cancel suitable powers of $\norm{Mf_{n}}_{p}$ on both sides and obtain the claim.
\end{proof}

There are other ways to deduce Theorem~\ref{thm:Doob-Lp} from Lemma~\ref{lem:Doob-weak}, such as real interpolation with the $L^{\infty}$ estimate.
But the above proof does not rely on the $L^{\infty}$ estimate, only on the $L^{1}$-like estimate in Lemma~\ref{lem:Doob-weak}.

It is often the case in martingale theory that inequalities near $L^{1}$ are the most powerful ones.
One can justify this by the observation that $L^{1}$ is the minimal assumption required to even define a martingale.

\neuevorl{2021-10-19}

If $f$ is a martingale, then the increments $df_{k}$ are orthogonal in the Hilbert space $L^{2}(\Omega)$.
Indeed, more generally, for any $n$ and $g \in L^{0}(\calF_{n-1})$, we have
\begin{equation}
\label{eq:increment-orthogonal-to-past}
\E(g df_{n})
= \E ( \E(g df_{n} | \calF_{n-1}))
= \E ( g \E( df_{n} | \calF_{n-1}))
= \E ( g \cdot 0)
= 0.
\end{equation}

\begin{definition}[Square function]
For a martingale $f$, we write
\[
Sf_{n} := \bigl( \sum_{k\leq n} \abs{df_{k}}^{2} \bigr)^{1/2}.
\]
\end{definition}
In continuous time, the analog of $Sf$ is denoted by $[f]$ and called the \emph{quadratic variation}.
We will construct it in \eqref{eq:covariation}.

Here and later, for notational simplicity, we will consider martingales with $f_{0}=0$.
Since martingale increments $df_{k}$ are orthogonal, we have
\[
\norm{ f_{n} }_{2} = \norm{ Sf_{n} }_{2}.
\]

Next, we will see that $S$ is bounded from $L^{1}$ to $L^{1,\infty}$.
The proof that we present for this fact uses a summation by parts identity of a kind that also appears e.g.\ in the \Ito{} formula.

\begin{lemma}
\label{lem:S-lookahead-stopped}
Let $f$ be a martingale, $\lambda>0$, and $\tau := \inf \Set{n \given \abs{f_{n}} > \lambda}$.
Then,
\[
\sum_{k} \E (\abs{df_{k}}^{2} \one_{\tau>k})
\leq
2 \lambda \norm{f}_{1}.
\]
\end{lemma}
On the left-hand side of the above estimate, we are taking the expectation of $S_{\tau-1}$.
Note that $\tau-1$ is not a stopping time.
Such predictability emulation will also be important in the BDG inequality below.

\begin{proof}
By the monotone convergence theorem, it suffices to consider the case $f_{N}=f_{N+1}=\dotsb$ for an arbitrarily large $N$.
Then, we can replace $\tau$ by $\tau \bmin N+1$ without changing the inequality.
This is necessary to apply the optional sampling theorem.

We use the identity
\[
\sum_{k<\tau} \abs{df_{k}}^{2} + \abs{f_{\tau-1}}^{2}
=
2 f_{\tau} f_{\tau-1} - 2 \sum_{k\leq \tau} f_{k-1} df_{k}.
\]
Note that
\[
\Set{k \leq \tau} = \Omega \setminus \Set{\tau<k} = \Omega \setminus \Set{\tau \leq k-1} \in \calF_{k-1}.
\]
Therefore, by \eqref{eq:increment-orthogonal-to-past}, we have
\[
\E ( \one_{k\leq \tau} f_{k-1} df_{k} )
=
0.
\]
Hence,
\[
\E \sum_{k<\tau} \abs{df_{k}}^{2} + \E \abs{f_{\tau-1}}^{2}
=
2 \E (f_{\tau} f_{\tau-1}).
\]
Therefore,
\[
\sum_{k} \E (\abs{df_{k}}^{2} \one_{\tau>k})
\leq
2 \E \abs{f_{\tau} f_{\tau-1}}
\leq
2 \E \abs{f_{\tau} \lambda}
\leq
2 \lambda \norm{f}_{1},
\]
where the last step follows from the optional sampling theorem, which gives in particular $f_{\tau} = \E_{\tau} f_{N+1}$.
\end{proof}

\begin{corollary}
Let $f$ be a martingale and $\lambda>0$.
Then,
\[
\abs{\Set{Sf>\lambda}}
\leq
3 \lambda^{-1} \norm{f}_{1}.
\]
\end{corollary}

\begin{proof}
Let $\tau := \inf \Set{n \given \abs{f_{n}} > \lambda}$.
Then
\[
\abs{\Set{Sf>\lambda}}
\leq
\abs{\Set{Sf_{\tau-1}>\lambda}}
+
\abs{\Set{\tau<\infty}}
\leq
\lambda^{-2} \norm{Sf_{\tau-1}}_{2}^{2}
+
\lambda^{-1} \norm{f}_{1}.
\]
By Lemma~\ref{lem:S-lookahead-stopped}, the first summand is $\leq 2 \lambda^{-1} \norm{f}_{1}$.
\end{proof}

\begin{lemma}[Davis decomposition]
\label{lem:Davis-decomposition}
For every martingale $(f_n)$, there exists a decomposition $f = f^{\pred} + f^{\bv}$ as a sum of two martingales such that $f^{\pred}$ has a predictable bound on jumps:
\begin{equation}
\label{eq:Davis-decomposition:good}
\abs{df^{\pred}_n} \leq 2 Mdf_{n-1},
\end{equation}
and $f^{\bv}$ has bounded variation:
\begin{equation}
\label{eq:Davis-decomposition:BV}
\E \sum_n \abs{df^{\bv}_n} \leq 2 \E Mdf.
\end{equation}
\end{lemma}
\begin{proof}
Let
\begin{align*}
d g_n
& :=
\min(1,\frac{Mdf_{n-1}}{\abs{df_n}}) df_n,\\
df^{\pred}_n 
&:=
d g_n - \E_{n-1}(d g_n),\\
d h_n 
&:= df_n - d g_n
= \max(0,1-\frac{Mdf_{n-1}}{\abs{df_n}}) df_n,\\
df^{\bv}_n
&:=
d h_n - \E_{n-1}(d h_n).
\end{align*}
Then, by definition,
\[
\abs{d g_n} \leq Mdf_{n-1},
\]
and, by positivity of conditional expectation, also
\[
\abs{\E_{n-1}(d g_n)}
\leq
\E_{n-1}(\abs{d g_n})
\leq
\E_{n-1}(Mdf_{n-1})
=
Mdf_{n-1}.
\]
This implies \eqref{eq:Davis-decomposition:good}.

On the other hand, we have the telescoping bound
\[
\abs{d h_n}
=
\max(0,\abs{df_n} - Mdf_{n-1})
=
Mdf_{n} - Mdf_{n-1},
\]
which implies \eqref{eq:Davis-decomposition:BV}.
\end{proof}

\begin{theorem}[Davis inequalities]
\label{thm:Davis-L1}
Let $f$ be a martingale with $f_0=0$.
Then,
\[
\E Sf \sim \E Mf.
\]
\end{theorem}
Here and later, $A \lesssim B$ means that $A\leq CB$ with an absolute constant $C$, and $A \sim B$ means that $A \lesssim B$ and $B \lesssim A$.
\begin{proof}
It suffices to consider martingales with $f_{N}=f_{N+1}=\dotsb=: f_{\infty}$ for some $N$, as long as we show bounds independent of $N$.

Let $f=f^{\pred}+f^{\bv}$ be a Davis decomposition as in Lemma~\ref{lem:Davis-decomposition}.
Then, for every $n>0$, we have the predictable bound
\[
\abs{f^{\pred}_n}
\leq \abs{f^{\pred}_{n-1}} + \abs{df^{\pred}_n}
\leq \abs{f^{\pred}_{n-1}}+2\abs{Mdf_{n-1}} =: \rho_{n-1}.
\]
Let $\lambda>0$ and $\tau := \inf\Set{n \given \rho_n > \lambda}$.
Then,
\begin{align*}
\abs{\Set{Sf^{\pred} > \lambda}}
&\leq
\abs{\Set{Sf^{\pred} > \lambda, \tau=\infty}} + \abs{\Set{\tau<\infty}}
\\ &\leq
\abs{\Set{Sf^{\pred,(\tau)} > \lambda}} + \abs{\Set{\tau<\infty}}
\\ &\leq
\lambda^{-2} \norm{Sf^{\pred,(\tau)}}_2^2 + \abs{\Set{\tau<\infty}}
\\ &=
\lambda^{-2} \norm{f^{\pred}_{\tau}}^2 + \abs{\Set{\tau<\infty}}.
\end{align*}
On the set $\Set{\tau=\infty}$, we use the bound $Mf^{\pred} \leq M\rho \leq \lambda$, while on the set $\Set{\tau<\infty}$, we use the bound $\abs{f^{\pred}_\tau} \leq \rho_{\tau-1} \leq \lambda$ (at this point, predictability is essential).
This gives
\begin{align*}
\abs{\Set{Sf^{\pred} > \lambda}}
&\leq
\lambda^{-2} \int_{\Set{\tau=\infty}} \abs{f^{\pred}_{\tau}}^2
+ \lambda^{-2} \int_{\Set{\tau<\infty}} \abs{f^{\pred}_{\tau}}^2
+ \abs{\Set{\tau<\infty}}
\\ &\leq
\lambda^{-2} \int_{\Set{Mf^{\pred} \leq \lambda}} \abs{Mf^{\pred}}^2
+ \abs{\Set{\tau<\infty}}
+ \abs{\Set{\tau<\infty}}.
\end{align*}
Note that $\Set{\tau<\infty} = \Set{M\rho>\lambda}$.
Inserting this in the above inequality and integrating in $\lambda$, we obtain
\begin{align*}
\E Sf^{\pred}
&=
\int_0^\infty \abs{\Set{Sf^{\pred} > \lambda}} \dif\lambda
\\ &\leq
\int_0^\infty \lambda^{-2} \int_{\Set{Mf^{\pred} \leq \lambda}} \abs{Mf^{\pred}}^2 \dif\lambda
+ 2 \int_0^\infty \abs{\Set{M\rho>\lambda}} \dif\lambda
\\ &=
\E \abs{Mf^{\pred}}^2 \int_{Mf^{\pred}}^\infty \lambda^{-2} \dif\lambda
+ 2 \E M\rho
\\ &=
\E Mf^{\pred}
+ 2 \E M\rho
\\ &\leq
3 \E Mf^{\pred}
+ 4 \E Mdf
\\ &\leq
3 \E Mf
+ 3 \E \sum_n \abs{df^{\bv}_n}
+ 4 \E Mdf
\\ &\leq
3 \E Mf
+ 10 \E Mdf.
\end{align*}
This, together with the simple bound
\[
\E S f^{\bv} \leq
\E \sum_{n} \abs{d f^{\bv}_{n}}
\leq 2 \E M df,
\]
implies $\E Sf \lesssim \E Mf$.

The proof of the converse inequality is similar and uses
\[
Sf^{\pred}_n
=
\bigl( (Sf^{\pred}_{n-1})^2 + (df^{\pred}_n)^2 )^{1/2}
\leq \abs{Sf^{\pred}_{n-1}} + \abs{df^{\pred}_n}
\leq \abs{Sf^{\pred}_{n-1}}+2\abs{Mdf_{n-1}} =: \rho_{n-1}.
\]
\end{proof}

\begin{lemma}[Garsia--Neveu]
\label{lem:Garsia-Neveu}
Let $W,Z$ be positive random variables such that, for every $\lambda>0$, we have
\begin{equation}
\label{eq:Garsia-Neveu:hypothesis}
\E (\one_{W>\lambda} (W-\lambda)) \leq \E(\one_{W>\lambda} Z).
\end{equation}
Then, for every $p\geq 1$, we have
\[
\norm{W}_p \leq p \norm{Z}_p.
\]
\end{lemma}
\begin{proof}
For $p=1$, it suffices to take $\lambda\to0$, so assume $p>1$.
Suppose first that $W$ is bounded by $\Lambda$ and \eqref{eq:Garsia-Neveu:hypothesis} holds for all $\lambda \in (0,\Lambda)$.
We will use the formula
\[
t^{p} = p(1-p) \int_{0}^{t} (t - \lambda) \lambda^{p-2} \dif \lambda.
\]
It yields
\begin{align*}
\E W^p
&=
p(p-1) \E \int_0^W (W-\lambda) \lambda^{p-2} \dif \lambda
\\ &=
p(p-1) \int_0^\Lambda \E ((W-\lambda) \one_{W>\lambda}) \lambda^{p-2} \dif \lambda
\\ &\leq
p(p-1) \int_0^\Lambda \E (Z \one_{W>\lambda}) \lambda^{p-2} \dif \lambda
\\ &=
p(p-1) \E \int_0^W Z \lambda^{p-2} \dif \lambda
\\ &=
p \E Z W^{p-1}
\\ &\leq
p (\E Z^p)^{1/p} (\E W^{p} )^{1-1/p}.
\end{align*}
Since $W$ is bounded, we can cancel a suitable power of $\E W^p$ on both sides and obtain the claim.

For general $W$, we apply the bounded case to $\min(W,\Lambda)$ and let $\Lambda\to\infty$.
\end{proof}

\begin{corollary}
\label{cor:Garsia-Neveu}
Let $(A_t)$ be an increasing predictable process with $A_0=0$ and $\xi$ a positive random variable such that, for everty $t$, we have
\[
\E_t(A_\infty - A_t) \leq \E_t(\xi).
\]
Then, for every $p \geq 1$, we have
\[
\norm{A_\infty}_p \leq p \norm{\xi}_p.
\]
\end{corollary}
\begin{proof}
For $\lambda>0$, let $\tau := \inf\Set{t \given A_{t+1} > \lambda}$.
Then,
\[
\E\max(A_\infty-\lambda,0)
=
\E (A_\infty - \lambda) \one_{\tau < \infty}
\leq
\E (A_\infty - A_\tau) \one_{\tau < \infty}
\leq
\E (\xi \one_{\tau < \infty})
\leq
\E (\xi \one_{A_\infty > \lambda}),
\]
so we can apply Lemma~\ref{lem:Garsia-Neveu}.
\end{proof}

\neuevorl{2021-10-26}

\begin{corollary}[Burkholder--Davis--Gundy inequalities]
\label{cor:BDG}
For every $p \in [1,\infty)$ and every martingale $f$, we have
\[
\norm{Sf}_{p} \sim \norm{Mf}_{p}.
\]
\end{corollary}
\begin{proof}
For any $t_0$ and $B \in \calF_{t_0}$, we can apply Theorem~\ref{thm:Davis-L1} to the martingale $\one_B (f-f^{(t_{0})})_{t\geq t_0}$, and obtain
\[
\E_{t_{0}} (M(f-f^{t_{0}}))
\sim
\E_{t_{0}} (S(f-f^{t_{0}})).
\]
Note that
\[
M f - M f_{t_{0}} \leq M(f - f^{t_{0}}) \leq 2Mf,
\quad
S f - S f_{t_{0}} \leq S(f-f^{t_{0}}) \leq Sf,
\]
so the conditions of Corollary~\ref{cor:Garsia-Neveu} hold for instance for $\tilde{A}_{t} = M f_{t}$ and $\xi = Sf$, with exception of the fact that $\tilde{A}_{t}$ is not predictable.
This can be remedied by taking $A_{t} = M f_{t-1}$ and $\xi = Sf + Mdf \leq 3 Sf$.
Corollary~\ref{cor:Garsia-Neveu} now shows that $\norm{Mf}_{p} \lesssim p \norm{Sf}_{p}$.

The proof of the converse inequality is similar.
\end{proof}

\subsection{Predictable square function}

\begin{definition}[Predictable square function]
For a martingale $f$, let
\[
s f_{n} := \bigl( \sum_{k=1}^{n} \E_{k-1}( \abs{df_{k}}^{2} ) \bigr)^{1/2}.
\]
\end{definition}
The continuous time version of $sf$ is denoted by $\<f\>$ and called the \emph{predictable quadratic variation}.
We mainly discuss $sf$  because the continuous time analog is sometimes used in stochastic analysis, and in order to present a few techniques.

The next lemma looks similar to $L^{p}$ contractivity of conditional expectation, but it cannot be proved by using this contractivity for each invidual summand.
\begin{lemma}
\label{lem:sum-of-Ek}
Let $z_{1},z_{2},\dotsc$ be positive random variables.
Then, for every $p\in [1,\infty)$, we have
\[
\E (\sum_{k=1}^{\infty} \E_{k-1}(z_{k}))^{p}
\leq
p^{p} \E (\sum_{k=1}^{\infty} z_{k})^{p}.
\]
\end{lemma}
\begin{proof}
The hypothesis of Corollary~\ref{cor:Garsia-Neveu} holds with equality for $A_{n}=W_{n}$ and $\xi=Z_{\infty}$, where
\begin{equation}
\label{eq:2}
W_{n} := \sum_{k=1}^{n} \E_{k-1}(z_{k}),
\quad
Z_{n} := \sum_{k=1}^{n} z_{k}.
\end{equation}
\end{proof}

\begin{corollary}
Let $p \in [2,\infty)$ and $f$ be a martignale.
Then, we have
\[
\norm{sf}_{p} \leq (p/2)^{1/2} \norm{Sf}_{p}.
\]
\end{corollary}

\begin{lemma}
\label{lem:truncation-concave-moment}
Let $Z,W$ be positive random variables such that, for some $C<\infty$ and all $\lambda>0$, we have
\[
\E(Z \bmin \lambda) \leq C \E(W \bmin \lambda).
\]
Then, for every $p \in (0,1]$, we have
\[
\E Z^{p} \leq C \E W^{p}.
\]
\end{lemma}
\begin{proof}
For $p=1$, it suffices to let $\lambda\to\infty$ in the hypothesis.

For $p<1$, we use the formula
\[
t^{p} = p(1-p) \int_{0}^{\infty} (t \bmin \lambda) \lambda^{p-2} \dif \lambda.
\]
Chaning the order of integration and applying the hypothesis for each $\lambda$, we obtain the claim.
\end{proof}

\begin{corollary}
Let $z_{1},z_{2},\dotsc$ be positive random variables.
Then, for every $p\in (0,1]$, we have
\[
\E (\sum_{k=1}^{\infty} z_{k})^{p}
\leq
2 \E (\sum_{k=1}^{\infty} \E_{k-1}(z_{k}))^{p}.
\]
\end{corollary}
Thus, we see that, for $p \in (0,2]$ and any martingale $f$, we have
\[
\norm{Sf}_{p} \lesssim_{p} \norm{sf}_{p}.
\]
\begin{proof}
Define $W,Z$ as in \eqref{eq:2}.
Let $\tau := \inf\Set{n \given W_{n+1} > \lambda}$.
Then,
\begin{align*}
\E (Z_{\infty} \bmin \lambda)
&\leq
\E (Z_{\tau} + \lambda\one_{\tau<\infty})
\\ &=
\sum_{k} \E (\one_{k \leq \tau} z_{k}) + \lambda \abs{\Set{W_{\infty}>\lambda}}
\\ &\leq
\sum_{k} \E (\one_{k \leq \tau} \E_{k-1}z_{k}) + \E(W_{\infty} \bmin \lambda)
\\ &\leq
2 \E(W_{\infty} \bmin \lambda).
\end{align*}
We conclude by Lemma~\ref{lem:truncation-concave-moment}.
\end{proof}
\begin{corollary}
Let $p \in (0,2]$ and $f$ be a martingale.
Then
\[
\norm{Mf}_{p} \leq 5^{1/p} \norm{sf}_{p}.
\]
\end{corollary}
\begin{proof}
Let $\tau := \inf\Set{ n \given sf_{n+1} > \lambda }$.
Then,
\begin{align*}
\E ( (Mf)^{2} \bmin \lambda^{2} )
&\leq
\E [ ( Mf^{\tau} )^{2} ]+ \lambda^{2} \E \one_{\tau<\infty}
\\ \text{(Doob's inequality)} &\leq
4 \E [ ( f^{\tau} )^{2} ] + \lambda^{2} \E \one_{sf_{\infty} > \lambda}
\\ &\leq
5 \E ((sf_{\infty})^{2} \bmin \lambda^{2}).
\end{align*}
We conclude by Lemma~\ref{lem:truncation-concave-moment}.
\end{proof}

\begin{lemma}[Good-$\lambda$ inequality]
\label{lem:good-lambda}
Let $f,g$ be positive random variables, $p \in (0,\infty)$, and suppose that we are given $\beta>1$ and $\delta,\epsilon \in \R_{>0}$ with $\beta^{p}\epsilon<1$.
Assume that, for every $\lambda>0$, we have
\[
\mu\Set{g>\beta\lambda, f \leq \delta\lambda} \leq \epsilon \mu\Set{g>\lambda}.
\]
Then,
\[
\E g^{p} \leq \frac{\delta^{-p}}{\beta^{-p}-\epsilon} \E f^{p}.
\]
\end{lemma}
\begin{proof}
The hypothesis implies
\begin{align*}
\mu\Set{g>\beta\lambda}
&=
\mu\Set{g>\beta\lambda, f \leq \delta\lambda}
+ \mu\Set{g>\beta\lambda, f > \delta\lambda}
\\ &\leq
\epsilon \mu\Set{g>\lambda}
+ \mu\Set{f > \delta\lambda}.
\end{align*}
Using the formula $t^{p} = p \int_{0}^{t} \lambda^{p-1} \dif \lambda$, we obtain
\begin{align*}
\E g^{p}
&=
p \E \int_{0}^{g} \lambda^{p-1} \dif \lambda
\\ &=
p \int_{0}^{\infty} \mu\Set{g>\lambda} \lambda^{p-1} \dif \lambda.
\end{align*}
Inserting the above estimate, we obtain
\[
\E (g/\beta)^{p} \leq \epsilon \E g^{p} + \E (f/\delta)^{p}.
\]
Rearranging, we obtain the claim.
\end{proof}

\begin{lemma}
Let $z_{1},z_{2},\dotsc$ be positive random variables.
Then, for every $p\in (0,\infty)$, we have
\[
\E (\sum_{k=1}^{\infty} z_{k})^{p}
\lesssim_{p}
\E (Mz \bmax \sum_{k=1}^{\infty} \E_{k-1}(z_{k}))^{p}.
\]
\end{lemma}
Thus, we see that, for $p \in [2,\infty)$ and any martingale $f$, we have
\[
\norm{Sf}_{p} \lesssim_{p} \norm{sf \bmax Mdf}_{p}.
\]
\begin{proof}
We continue using notation \eqref{eq:2}.
We will verify the hypothesis of Lemma~\ref{lem:good-lambda}.
Let $\beta,\delta \in \R_{>0}$ with $\beta > \delta + 1$.
Consider the stopping times
\[
\tau := \inf\Set{n \given Z_{n}>\lambda},
\tau_{\beta} := \inf\Set{n \given Z_{n}>\beta\lambda},
\sigma := \inf\Set{ n \given (W_{n+1}\bmax z_{n}) > \delta\lambda }
\]
and the process
\[
h_{n} = \sum_{k\leq n} \one_{\tau < k \leq \tau_{\beta}\bmin \sigma} z_{k}.
\]
Then, if $Z>\beta\lambda$ and $W \bmax Mz \leq \delta\lambda$, we have $\sigma=\infty$, so that
\[
h_{\infty} = \sum_{\tau < k \leq \tau_{\beta}} z_{k}
\geq
\beta\lambda - \sum_{k < \tau} z_{k} - z_{\tau}
\geq
\beta\lambda - \lambda - \delta\lambda
=
(\beta-1-\delta) \lambda.
\]
Hence,
\begin{align*}
\mu\Set{Z > \beta\lambda, W \bmax Mz \leq \delta\lambda}
&\leq
\mu\Set{h_{\infty} > (\beta-1-\delta) \lambda}
\\ &\leq
\frac{1}{(\beta-1-\delta) \lambda} \E h_{\infty}
\\ &=
\frac{1}{(\beta-1-\delta) \lambda} \sum_{k} \E ( \one_{\tau < k \leq \tau_{\beta}\bmin \sigma} z_{k} )
\\ &=
\frac{1}{(\beta-1-\delta) \lambda} \sum_{k} \E ( \one_{\tau < k \leq \tau_{\beta}\bmin \sigma} \E_{k-1} z_{k} )
\\ &\leq
\frac{1}{(\beta-1-\delta) \lambda} \E ( \one_{\tau < \infty} \delta \lambda )
\\ &=
\frac{\delta}{\beta-1-\delta} \mu\Set{Z>\lambda}.
\end{align*}
Now it suffices to take any $\beta>1$ and $\delta$ sufficiently small depending on $\beta$.
\end{proof}

\begin{remark}
Most of the material in this section is from \cite{MR365692}.
\end{remark}

\neuevorl{2021-11-02}


\section{L\'epingle inequality}

For $0<r<\infty$ and a sequence of random variables $f = (f_{n})_{n}$, the $r$-variation of $f$ on the interval $[t',t]$ is defined by
\begin{equation}
\label{eq:Vr}
V^{r} f_{t',t} :=
\sup_{J, t' \leq u_{0} < \dotsb < u_J \leq t}
\Bigl( \sum_{j=1}^J \dist{f_{u_{j-1}}}{f_{u_{j}}}^{r} \Bigr)^{1/r},
\end{equation}
where the supremum is taken over arbitrary increasing sequences.
Analogously, $V^\infty f_{t',t} := \sup_{t' \leq u' < u \leq t} \dist{f_{u'}}{f_{u}}$.

We abbreviate $V^r f_t := V^r f_{0,t}$ and $V^r f := V^r f_{0,\infty}$.

The following result is a quantitative version of the martingale convergence theorem.
\begin{theorem}[L\'epingle inequality]
\label{thm:lepingle}
For every $p \in [1,\infty)$, there exists a constant $C_{p}<\infty$ such that, for every $r>2$, and every martingale $f$, we have
\begin{equation}
\label{eq:lepingle}
\norm{ V^{r}f }_{L^{p}}
\leq C_{p} \frac{r}{r-2}
\norm{ Mf }_{L^{p}}.
\end{equation}
\end{theorem}

The main difficulty here is that the supremum in the definition of $V^r$ is taken over not necessarily adapted partitions.
We will remedy this by a greedy selection algorithm, stopping as soon as we see a large jump.
We have to use stopping rules depending on a parameter $m$ in order to capture jumps of all possible sizes.

\begin{proof}
By the monotone convergence theorem, we may assume that $f$ becomes constant after some time $N$.

Let $M_{t} := \sup_{t'' \leq t' \leq t} \dist{f_{t'}}{f_{t''}}$.
For each $m\in \N$, define an adapted partition by
\begin{equation}
\label{eq:stop-times}
\stoptime{m}{0} := 0,
\quad
\stoptime{m}{j+1} := \inf \Set{ t \geq \stoptime{m}{j} \given \dist{ f_{t} }{ f_{\tau_{j}} } \geq 2^{-m} M_{t}, M_{t}>0 }.
\end{equation}

\begin{claim}
\label{lem:comparable-jump}
Let $0 \leq t' < t < \infty$ and $m\geq 2$.
Suppose that
\begin{equation}
\label{eq:jump-magnitude}
2 < \dist{ f_{t'} }{ f_{t}}/(2^{-m}M_{t}) \leq 4.
\end{equation}
Then there exists $j$ with $t' < \stoptime{m}{j} \leq t$ and
\begin{equation}
\label{eq:comparable-jump}
\dist{ f_{t'} }{ f_{t}}
\leq
8 \dist{ f_{\stoptime{m}{j-1}} }{ f_{\stoptime{m}{j}}}.
\end{equation}
\end{claim}
\begin{proof}[Proof of the claim]
Let $j$ be the largest integer with $\tau' := \stoptime{m}{j} \leq t$.
We claim that $\tau' > t'$.
Suppose for a contradiction that $\tau' < t'$ (the case $\tau'=t'$ is similar but easier).
By the hypothesis \eqref{eq:jump-magnitude} and the assumption that $t,t'$ are not stopping times, we obtain
\[
2 \cdot 2^{-m} M_{t}
<
\dist{ f_{t'}}{ f_{t} }
\leq
\dist{ f_{\tau'}}{ f_{t'}} + \dist{ f_{\tau'}}{ f_{t}}
<
2^{-m} M_{t'} + 2^{-m} M_{t}
\leq
2 \cdot 2^{-m} M_{t},
\]
a contradiction.
This shows $\tau' > t'$.

It remains to verify \eqref{eq:comparable-jump}.
Assume that $M_{\tau'} < M_{t}/2$.
Then, for some $\tau'' \in (\tau, t]$, we have $\dist{ f_{\tau'}}{ f_{\tau''}} \geq M_{t}/2 \geq 2^{-m} M_{\tau''}$, contradicting maximality of $\tau'$.
It follows that
\[
\dist{ f_{\stoptime{m}{j-1}} }{ f_{\stoptime{m}{j}} }
\geq
2^{-m} M_{\tau'}
\geq
2^{-m} M_{t}/2
\geq
\dist{ f_{t'} }{ f_{t}}/8.
\qedhere
\]
\end{proof}

Next, for any $0<\rho<r<\infty$, we will show the pathwise inequality
\begin{equation}
\label{eq:ptw-domination}
V^{r}_{t}(f_{t})^{r}
\leq
8^{\rho} \sum_{m=2}^{\infty} \bigl( 2^{-(m-2)} M_{\infty} \bigr)^{r-\rho} \sum_{j=1}^{\infty}
\dist{ f_{\stoptime{m}{j-1}} }{ f_{\stoptime{m}{j}} }^{\rho}.
\end{equation}
Let $(u_{l})$ be any increasing sequence.
For each $l$ with $\dist{ f_{u_{l}} }{ f_{u_{l+1}}}\neq 0$, let $m=m(l) \geq 2$ be such that
\[
2 < \dist{  f_{u_{l}} }{ f_{u_{l+1}} }/(2^{-m}M_{u_{l+1}}) \leq 4.
\]
Such $m$ exists because the distance is bounded by $M_{u_{l+1}}$.

Let $j=j(l)$ be given by the Claim above with $t'=u_{l}$ and $t=u_{l+1}$.
Then
\[
\dist{  f_{u_{l}} }{ f_{u_{l+1}} }^{r}
\leq
8^{\rho} \dist{ f_{\stoptime{m}{j-1}} }{ f_{\stoptime{m}{j}}}^{\rho}
\cdot (4 \cdot 2^{-m} M_{u_{l+1}})^{r-\rho}.
\]
Since each pair $(m,j)$ occurs for at most one $l$, this implies
\[
\sum_{l} \dist{  f_{u_{l}} }{ f_{u_{l+1}} }^{r}
\leq
8^{\rho} \sum_{m,j} \dist{ f_{\stoptime{m}{j-1}} }{ f_{\stoptime{m}{j}}}^{\rho}
\cdot (2^{-(m-2)} M_{\infty})^{r-\rho}.
\]
Taking the supremum over all increasing sequences $(u_{l})$, we obtain \eqref{eq:ptw-domination}.

Since we assumed that $f_{n}$ becomes independent of $n$ for sufficiently large $n$, we have
\[
M_{\infty} \leq V^{r}_{t}(f_{t}) < \infty.
\]
Substituting this inequality in \eqref{eq:ptw-domination} and canceling $V^{r}_{t}(f_{t})^{r-2}$ on both sides, we obtain
\begin{equation}
\label{eq:ptw-domination:only-square}
V^{r}_{t}(f_{t})^{\rho}
\leq
8^{\rho} \sum_{m=2}^{\infty} 2^{-(m-2)(r-\rho)} \sum_{j=1}^{\infty}
\dist{ f_{\stoptime{m}{j-1}} }{ f_{\stoptime{m}{j}} }^{\rho}.
\end{equation}

By the optional sampling theorem, for each $m$, the sequence $(f_{\stoptime{m}{j}})_{j}$ is a martingale with respect to the filtration $(\calF_{\stoptime{m}{j}})_{j}$.
Let
\[
S_{(m)} := \Bigl( \sum_{j=1}^{\infty} \dist{f_{\stoptime{m}{j-1}}}{f_{\stoptime{m}{j}}}^{2} \Bigr)^{1/2}
\]
denote the square function of this martingale.
Then \eqref{eq:ptw-domination:only-square} implies
\[
V^{r}f
\leq
8 \Bigl( \sum_{m=2}^{\infty} 2^{-(m-2)(r-2)} S_{(m)}^{2} \Bigr)^{1/2}
\leq
8 \sum_{m=2}^{\infty} 2^{-(m-2)(r-2)/2} S_{(m)}
\]
By Minkowski's inequality, this implies
\[
\norm{V^{r}f}_{L^{p}}
\leq
8 \sum_{m=2}^{\infty} 2^{-(m-2)(r-2)/2} \norm{S_{(m)}}_{L^{p}}.
\]
Applying the BDG inequality for the martingales $(f_{\stoptime{m}{j}})_{j}$, we obtain
\begin{align*}
\norm{ V^{r}f }_{L^{p}}
&\lesssim_{p}
\sum_{m=2}^{\infty} 2^{-(m-2)(r-2)/2} \norm{ \sup_{j} \abs{f_{\stoptime{m}{j}}} }_{L^{p}}
\\ &\lesssim
\frac{r}{r-2} \norm{ M f }_{L^{p}}.
\qedhere
\end{align*}
\end{proof}

\begin{remark}
The dependence of the constant in Theorem~\ref{thm:lepingle} on the variation exponent $r$ can be improved to $\sqrt{\frac{r}{r-2}}$ using a vector-valued BDG inequality.
In fact, a slightly more careful version of the above argument already shows this for $p \in [2,\infty)$.
\end{remark}

\subsection{Martingale convergence}

\begin{corollary}[Martingale convergence]
\label{cor:L1-mart-convergence}
Let $f$ be a martingale with $\norm{f}_1 < \infty$.
Then, for every $r>2$, $V^r f$ is finite a.s.
In particular, the sequence $(f_n)_n$ converges a.s.
\end{corollary}

\begin{proof}
Fixing $\lambda<\infty$ and $N\in\N$, consider the stopping time
\[
\tau := \inf \Set{ n \given \abs{f_n} > \lambda }.
\]
Then $M f^{\tau \bmin N} \leq \lambda \bmax \abs{f_{\tau \bmin N}}$.
Since $\tau \bmin N$ is a bounded stopping time, by the optional sampling theorem, we have
\[
\norm{M f^{\tau \bmin N}}_1
\leq
\lambda + \norm{f_{\tau \bmin N}}_1
\leq
\lambda + \norm{f}_1.
\]
By L\'epingle's inequality, for $r>2$, we obtain
\[
\norm{V^r f^{\tau \bmin N}}_1
\lesssim
\lambda + \norm{f}_1.
\]
Since the right-hand side does not depend on $N$ and by the monotone convergence theorem, we obtain
\[
\norm{V^r f^{\tau}}_1
\lesssim
\lambda + \norm{f}_1.
\]
In particular, the function $V^r f^\tau$ is finite a.s.
On the set $\Set{Mf < \lambda}$, we have $f=f^\tau$, so that $V^r f$ is finite a.s.\ on this set.

Taking the union over $\lambda\in\N$, we find that $V^r f$ is finite a.s.\ on the set $\Set{Mf < \infty}$.
By Doob's maximal inequality, this set has full probability.
\end{proof}

We have shown that, for an $L^1$ martingale $f$, the limit $f_\infty := \lim_{n\to\infty} f_n$ exists a.s.
We recall the condition under which this convergence also holds in $L^{1}$.
We start with an example that shows that this is not always the case.

\begin{example}
\label{ex:L1-not-unif-integr}
Consider $\Omega = [0,1]$ with the dyadic filtration.
The sequence
\[
f_n = 2^n \one_{[0,2^{-n}]}
\]
is a martingale and converges pointwise a.s.\ to $0$.
On the other hand, $\E f_n = 1 \neq 0 = \E 0$.
\end{example}

\begin{definition}
A set $F$ of real random variables is called \emph{uniformly integrable} if
\[
\lim_{\lambda\to\infty} \sup_{f\in F} \E(\one_{\abs{f}>\lambda} \abs{f}) = 0. 
\]
\end{definition}

\begin{example}
If $f\in L^1$, then the singleton $\Set{f}$ is uniformly integrable.
\end{example}

Uniform integrability is useful because it is the minimal hypothesis under which an analogue of the dominated convergence theorem holds.
We recall this analogue.

\begin{proposition}[Dominated convergence]
For every sequence of random variables $f = (f_n)$, the following are equivalent:
\begin{enumerate}
\item $f$ is Cauchy in $L^1$,
\item $f$ is uniformly integrable and Cauchy in probability, that is,
\[
(\forall \epsilon>0) (\exists N\in\N) (\forall m,n\geq N) \mu\Set{\dist{f_{n}}{f_{m}}>\epsilon} < \epsilon.
\]
\end{enumerate}
\end{proposition}

In relation with martingales, it is important that uniform integrability is preserved under conditional expectation.
This can be seen using the following characterization.

\begin{lemma}
Let $F$ be a set of random variables.
The following are equivalent.
\begin{enumerate}
\item $F$ is uniformly integrable.
\item $\sup_{f\in F} \E\abs{f} < \infty$ and
\begin{equation}
\label{eq:unif-integr-muA}
\lim_{\mu(A)\to 0} \sup_{f\in F} \E(\one_A \abs{f}) = 0.
\end{equation}
\end{enumerate}
\end{lemma}
\begin{proof}
Without loss of generality, all r.v.\ in $F$ are positive.

$\implies$:
For every $\lambda>0$, we have
\[
\E f
\leq
\lambda + \E (\one_{f>\lambda} f).
\]
Taking the supremum over $f\in F$ on both sides and choosing $\lambda$ such that
\[
\sup_{f\in F} \E (\one_{f>\lambda} f) < \infty,
\]
we see that $\sup_{f\in F} \E\abs{f} < \infty$.

Moreover,
\[
\E(\one_A f)
\leq
\mu(A) \lambda + \E(\one_{f>\lambda} f).
\]
Taking first the supremum over $f\in F$, and then $\lim_{\mu(A)\to 0}$, we obtain
\[
\lim_{\mu(A)\to 0} \sup_{f\in F}
\E(\one_A f)
\leq
\sup_{f\in F} \E(\one_{f>\lambda} f).
\]
Taking $\lim_{\lambda\to 0}$, we obtain \eqref{eq:unif-integr-muA}.

$\impliedby$:
We have $\mu\Set{f>\lambda} \leq \lambda^{-1} \E f$.
Taking the supremum over $f\in F$, we obtain
\[
\sup_{f\in F} \mu\Set{f>\lambda} \leq \lambda^{-1} \sup_{f\in F} \E f.
\]
The RHS converges to $0$ as $\lambda\to\infty$.
The second hypothesis now implies that $F$ is uniformly integrable.
\end{proof}

\begin{lemma}
\label{lem:E-unif-integr}
Let $F \subset L^1(\Omega,\calF)$ be uniformly integrable.
Then the set
\[
\Set{ \E(f|\calF') \given f\in F, \calF'\subseteq\calF \text{ sub-$\sigma$-algebra} }
\]
is uniformly integrable.
\end{lemma}
\begin{proof}
We verify the characterization above.
We have
\[
\E \abs{ \E(f|\calF') }
\leq
\E \E( \abs{f} |\calF')
= \E \abs{f}.
\]
This verifies the first condition.

Fixing $\lambda>0$, we also have
\[
\E ( \one_A \abs{ \E(f|\calF') } )
\leq
\E ( \one_A \E( \abs{f} |\calF') )
=
\E ( \E(\one_A|\calF') \abs{f} )
\leq
\lambda \E \abs{f} + \E( \one_{\E(\one_A|\calF') > \lambda} \abs{f} )
\]
Since $\mu\Set{ \E(\one_A|\calF') > \lambda } \leq \lambda^{-1} \mu(A)$ and using \eqref{eq:unif-integr-muA} for the set $F$, we obtain
\[
\lim_{\mu(A)\to 0} \sup_{f\in F, \calF'\subseteq\calF}
\E ( \one_A \abs{ \E(f|\calF') } )
\leq
\lambda \E \abs{f}.
\]
Since $\lambda>0$ was arbitrary, this implies \eqref{eq:unif-integr-muA} for the set of conditional expectations.
\end{proof}

The above results can now be summarized as follows.
\begin{proposition}[Martingale closure]
\label{prop:unif-integrable-mart-closure}
For every martingale $f=(f_{n})_{n\in\N}$, the following are equivalent.
\begin{enumerate}
\item The set $\Set{f_{n}}_{n\in\N}$ is uniformly integrable.
\item The sequence $(f_{n})_{n\in\N}$ converges in $L^{1}$.
\item There exists $f_{\infty}\in L^{1}$ such that, for every $n\in\N$, we have $f_{n} = \E_{n}f_{\infty}$.
\end{enumerate}
If the above conditions hold, then we can take $f_\infty = \lim_{n\to\infty} f_n$ a.s.\ and in $L^{1}$.
\end{proposition}

\neuevorl{2021-11-09}

\begin{remark}
We have seen that, for a martingale $f$, we have
\[
Mf \in L^{1} \implies f \text{ is uniformly integrable} \implies f \in L^{1}.
\]
The second implication cannot be reversed by Example~\ref{ex:L1-not-unif-integr}.
The following example shows that the first implication also cannot be reversed, and also that $M$ is not bounded on $L^{1}$.
\end{remark}
\begin{example}
Let $\Omega=[0,1]$ with the dyadic filtration and
\[
f_{\infty} := \sum_{m=0}^{\infty} (m+1)^{-2} 2^{m} \one_{[2^{-m-1},2^{-m}]}
\]
and $f_{n} :=\E_{n} f_{\infty}$.
Then, for $x \in (2^{-l-1},2^{-l})$, we have
\begin{multline*}
Mf(x) = f_{l}(x)
= 2^{l} \sum_{m=l}^{\infty} \int_{0}^{2^{-l}} (m+1)^{-2} 2^{m} \one_{[2^{-m-1},2^{-m}]}(y) \dif y
\\= 2^{l-1} \sum_{m=l}^{\infty} (m+1)^{-2}
\sim (l+1)^{-1} 2^{l}.
\end{multline*}
We see that $Mf \not\in L^{1}$.
\end{example}


\section{Vector-valued inequalities}

\begin{remark}
The norm on a Banach space $X$ will be sometimes denoted by $Xf := \norm{f}_{X}$.
This is consistent with operator notation, in that every norm is also a sublinear operator.
\end{remark}

\begin{lemma}\label{lem:vv-BDG}
Let $X$ be a Banach space.
Let $T$ be a subadditive operator that maps $X$-valued martingales starting at $0$ to $\R_{\geq 0}$-valued increasing adapted processes starting at $0$ such that, for every stopping time $\tau$, we have
\[
T f_{\tau} = T (f^{\tau})_{\tau},
\]
and for every time $t$ we have
\[
T f_{t} \leq T f_{t-1} + X d f_{t}.
\]
Let $U$ be another operator satisfying the same properties as $T$.

Assume that, for some $r\in (1,\infty)$ and every $X$-valued martingale $f$, we have
\begin{equation}\label{eq:weak-X-BDG}
\bbP \Set{ T f > \lambda }
\lesssim
\lambda^{-r} (L^{r} U f)^{r}.
\end{equation}
Then, for every $q \in [1,\infty)$ and every $X$-valued martingale $f$, we have
\begin{equation}\label{eq:X-BDG}
L^{q} T f
\lesssim_{q}
L^{q} U f + L^{q} M X d f.
\end{equation}
\end{lemma}
The formulation is flexible enough to apply to any combination $\Set{T,U} = \Set{X M,X S}$, and also some other operators, such as martingale transforms or $r$-variation norms.

\begin{proof}
Consider first the case $q=1$.
We use the Davis decomposition $f = f^{\pred} + f^{\bv}$.
Of the two possible generalizations of the Davis decomposition to $X$-valued martingales, we use the one in which the absolute value is replaced by the $X$-norm (the alternative would be to make a decomposition for each $k$ separately).
Specifically, we have
\[
X d f^\pred_n \leq 2 M X d f_{n-1},
\quad
\E \sum_n X d f^\bv_n \leq 2 \E M X d f.
\]
For $\lambda > 0$, define the stopping time
\[
\tau := \inf \Set{ t \given U f^{\pred}_{t}>\lambda \text{ or } M X df_{t}>\lambda }.
\]
We claim that
\begin{equation}
\label{eq:stopped-pred}
U f^{\pred}_{\tau} \leq U f^{\pred} \wedge 3 \lambda.
\end{equation}
Indeed, the first bound is trivial.
The second bound is trivial if $\tau=\infty$, so assume $\tau \in (0,\infty)$.
Then, by properties of the Davis decomposition, we have
\[
U f^{\pred}_{\tau}
\leq
U f^{\pred}_{\tau-1}
+
X d f^{\pred}_\tau
\leq
\lambda
+
2 M X d f_{\tau-1}
\leq
3 \lambda.
\]
Also,
\begin{align*}
\Set{ T f^{\pred} > \lambda }
&\subseteq
\Set{ T f^{\pred}_{\tau} > \lambda }
\cup
\Set{ \tau < \infty }
\\ &\subseteq
\Set{ T f^{\pred}_{\tau} > \lambda }
\cup
\Set{ U f^{\pred} > \lambda }
\cup
\Set{ M X df_\tau > \lambda }.
\end{align*}
By the layer cake formula,
\begin{align*}
\norm{T f^{\pred}}_{L^{1}}
&=
\int_{0}^{\infty} \bbP \Set{ T f^{\pred} > \lambda } \dif\lambda
\\ &\leq
\int_{0}^{\infty} \bbP \Set{ T f^{\pred}_{\tau} > \lambda } \dif\lambda
+
\int_{0}^{\infty} \bbP \Set{ U f^{\pred} > \lambda } \dif\lambda
\\ & \quad +
\int_{0}^{\infty} \bbP \Set{ M X df > \lambda } \dif\lambda
=: I + II + III.
\end{align*}
The term $III$ is easy to bound by the layer cake formula.

In the term $II$, we use the properties of the Davis decomposition to estimate
\begin{align*}
II &=
\norm{ U f^{\pred} }_{L^{1}}
\\ &\leq
\norm{ U f }_{L^{1}}
+
\norm{ U f^{\bv} }_{L^{1}}
\\ &\leq
\norm{ U f }_{L^{1}}
+
\norm{ \sum_{n} X d f^{\bv}_{n} }_{L^{1}}
\\ &\leq
\norm{ U f }_{L^{1}}
+
2 \norm{ M X d f }_{L^{1}}.
\end{align*}
Using the hypothesis \eqref{eq:weak-X-BDG} and \eqref{eq:stopped-pred}, we bound the first term by
\begin{align*}
I &\lesssim
\int_{0}^{\infty} \lambda^{-r} \norm{ U f^{\pred}_{\tau} }_{L^{r}}^{r} \dif\lambda
\\ &\leq
\int_{0}^{\infty} \lambda^{-r} \norm{ U f^{\pred} \wedge 3\lambda }_{L^{r}}^{r} \dif\lambda
\\ &=
\bbE \int_{0}^{\infty} \min\bigl(\lambda^{-r} (U f^{\pred} )^{r}, 3^r \bigr) \dif\lambda
\\ &\lesssim
\bbE U f^{\pred}
= II,
\end{align*}
and we reuse the previously established estimate for $II$.

We have shown \eqref{eq:X-BDG} for $q=1$, and we will now extend this claim to $q>1$.
In doing so, we may replace $U f$ by $(U f + M X d f)/2$.
This has the effect that we may omit the second summand on the RHS of \eqref{eq:X-BDG}.

For $q>1$, we use the Garsia-Neveu lemma.
For a martingale $f$, let $\startat{\tau}{f} := f - f^{\tau}$ be the martingale $f$ started at (stopping) time $\tau$.

For any $t \in \N$ and $B \in \calF_{t}$, we can apply the case $q=1$ to the martingale $\one_B \startat{t}{f}$, and obtain
\[
\E_{t} T \startat{t}{f}
\lesssim
\E_{t} U \startat{t}{f}.
\]
Note that
\[
T f_{\infty} - T f_{t} = T f_{\infty} - T (f^{t})_{\infty} \leq T(f - f^{t})_{\infty} = T \startat{t}{f}_{\infty},
\]
\[
U \startat{t}{f}_{\infty} \leq (U f + U f^{t})_{\infty} \leq 2 U f_{\infty},
\]
so the conditions of Corollary~\ref{cor:Garsia-Neveu} hold for $A_{t} = T f_{(t-1) \bmax 0}$ and $\xi \sim U f_{\infty}$.
Corollary~\ref{cor:Garsia-Neveu} now shows the claim.
\end{proof}
Monotonicity of $Uf$ is used in estimates for both $I$ and $II$ above.

\begin{corollary}[$\ell^{r}$ valued BDG inequality]
\label{cor:lr-BDG}
For every $k\in\N$, let $f_{k} = (f_{k,n})_{n}$ be a real-valued martingale starting at $0$.
Then, for every $q,r \in [1,\infty)$, we have
\[
L^{q} \ell^{r}_{k} M f_{k}
\lesssim_{q}
L^{q} \ell^{r}_{k} S f_{k}.
\]
\end{corollary}
\begin{proof}
For $r=q$, this follows from the scalar-valued Doob's inequality and Fubini.
For general $q$, we apply Lemma~\ref{lem:vv-BDG} with $X=\ell^{r}$.
\end{proof}

Toy example: better constant in L\'epingle.

\subsection{Weighted and vector-valued Doob inequalities}
Let $w \in L^{1}(\Omega,\calF_{\infty})$ be a positive function, that we will call a \emph{weight}.
We write $w_{k} := \E_{k} w$, $w^{*} := \sup_{k} w_{k}$, and $w B := \int_{B} w \dif\mu$ for $B\in\calF_{\infty}$.

\begin{lemma}
\label{lem:Doob-weak-weighted}
Let $f$ be an $\R_{\geq 0}$-valued submartingale (recall that this means $f_{k} \leq E_{k} f_{n}$ for $k\leq n$).
Then, for every $\lambda>0$ and $n\in\N$, we have
\[
\lambda w \Set{Mf_{n} > \lambda}
\leq
\int_{\Set{Mf_{n}>\lambda}} f_{n} Mw_n \dif\mu.
\]
\end{lemma}
\begin{proof}
For a fixed $\lambda$, define the stopping time
\[
\tau := \inf \Set{k \given f_{k} > \lambda}.
\]
Then, $\Set{Mf_{n} > \lambda} = \Set{\tau \leq n}$.
Hence,
\begin{align*}
\lambda w \Set{Mf_{n} > \lambda}
&=
\lambda \sum_{k \leq n} w \Set{\tau = k}
\\ &\leq
\sum_{k \leq n} \int_{\Set{\tau = k}} f_{k} w_{k} \dif\mu
\\ \text{(submartingale property)} &\leq
\sum_{k \leq n} \int_{\Set{\tau = k}} f_{n} w_{k} \dif\mu
\\ &\leq
\int_{\Set{Mf_{n} > \lambda}} f_{n} w^{*} \dif\mu.
\qedhere
\end{align*}
\end{proof}

We recall that, for a function $f$ on a measure space $(\Omega,\mu)$ and $p\in [1,\infty)$, the $L^{p,\infty}$ quasinorm is defined by
\[
\norm{ f }_{L^{p,\infty}}
:=
\sup_{\lambda>0} \lambda \mu\Set{\abs{f}>\lambda}^{1/p}.
\]
For $p=\infty$, we have $\norm{ f }_{L^{p,\infty}}= \norm{ f }_{L^{\infty}}$.

\begin{theorem}[Marcinkiewicz/real interpolation, see e.g.\ {\cite{MR3243734}}]
Let $(\Omega,\mu)$ and $(\tilde\Omega,\tilde\mu)$ be $\sigma$-finite measure spaces.
Let $0 < p_0 < p_1 \leq \infty$.
Let
\[
T : L^{p_0}(\Omega) + L^{p_1}(\Omega)
\to
L^{0}(\tilde\Omega)
\]
be a subadditive operator such that, for every function $f$, we have
\[
\norm{ Tf }_{L^{p_j,\infty}(\tilde\Omega)}
\leq A_j
\norm{ f }_{L^{p_j}(\Omega)},
\quad j=0,1.
\]
Then, for every $p\in (p_0,p_1)$, we have
\[
\norm{ Tf }_{L^{p}(\tilde\Omega)}
\lesssim_p
\norm{ f }_{L^{p}(\Omega)}.
\]
\end{theorem}


\begin{lemma}
Let $p \in (1,\infty)$, $f$ be a submartingale, and $w$ a weight.
Then,
\[
\norm{M f}_{L^p(w)}
\lesssim_p
\norm{f}_{L^p(w^*)}
\]
\end{lemma}
\begin{proof}
We will show the following more precise estimate:
\begin{equation}
\label{eq:Lpw-max-ineq-MwN}
\norm{Mf_{N}}_{L^{p}(w)}
\lesssim_{p}
\norm{f_N}_{L^p(Mw_N)},
\end{equation}
from which the claim follows by letting $N\to\infty$.
The main advantage of this formulation is that the right-hand side increases in $N$, which would not be the case if we would use the weight $w^{*}$ instead of $Mw_{N}$.

Lemma~\ref{lem:Doob-weak-weighted} implies in particular
\[
\norm{M f_{N}}_{L^{1,\infty}(w)}
\lesssim
\norm{f_{N}}_{L^1(Mw_{N})}.
\]
The estimate
\[
\norm{M f_{N}}_{L^{\infty}(w)}
\lesssim
\norm{f_{N}}_{L^\infty(Mw_{N})}
\]
is easy to see.
By real interpolation, these two bounds imply \eqref{eq:Lpw-max-ineq-MwN}.
\end{proof}

\begin{proposition}[Vector-valued maximal inequality]
Let $p \in (1,\infty)$ and $r \in (1,\infty]$.
Then, for any sequence of martingales $f_{k}=(f_{k,n})_{n}$, we have
\[
L^{p} \ell^{r}_{k} M f_{k}
\lesssim_{q,r}
L^{p} \ell^{r}_{k} f_{k}.
\]
\end{proposition}
\begin{proof}
The case $r=\infty$ follows from the scalar case, because $\ell^{\infty} f_{k}$ is a submartingale.

The case $p=r$ also follows from the scalar case and Fubini.

For $1<r<p<\infty$, let $\rho := (p/r)' = p/(p-r)$.
Then, for some weight $w \in L^{\rho}$ with $\norm{w}_{\rho} \leq 1$, we have
\begin{align*}
\norm{ \ell^{r}_{k} M f_{k} }_{p}^{r}
&\lesssim
\Bigl( \E \bigl( \sum_{k} (M f_{k})^{r} \bigr)^{p/r} \Bigr)^{r/p}
\\ &=
\E \sum_{k} (M f_{k})^{r} w
\\ &\lesssim
\E \sum_{k} \abs{f_{k}}^{r} w^{*}
\end{align*}
and we conclude using Hölder's inequality and boundedness of the maximal operator on $L^{\rho}$.

For $1 < p < r < \infty$, let $s \in (1,p)$.
Then, for some sequence of weights with
\[
\norm{ \ell^{(r/s)'}_{k} w_{k} }_{(p/s)'} \sim 1,
\]
we have
\begin{align*}
\norm{ \ell^{r}_{k} M f_{k} }_{p}^{s}
&=
\norm{ \ell^{r/s}_{k} (M f_{k})^s }_{p/s}
\\ &\sim
\int \sum_k w_k (M f_k)^s
\\ &\lesssim
\int \sum_k w_k^* \abs{f_k}^s
\\ &\leq
\norm{ \ell^{(r/s)'}_{k} w_{k}^* }_{(p/s)'}
\norm{ \ell^{r/s}_{k} \abs{f_{k}}^s }_{p/s}.
\end{align*}
Since $(r/s)' < (p/s)'$, from the previously shown case, we obtain
\[
\norm{ \ell^{(r/s)'}_{k} w_{k}^* }_{(p/s)'}
\lesssim
\norm{ \ell^{(r/s)'}_{k} w_{k} }_{(p/s)'}
\sim 1.
\qedhere
\]
\end{proof}

\begin{remark}
Most of this section is from \cite[Section 3]{MR3617205}.
\end{remark}

\neuevorl{2021-11-16}


\newcommand{\tX}{\tilde{X}}
\newcommand{\tY}{\tilde{Y}}
\newcommand{\tZ}{\tilde{Z}}
\newcommand{\tBX}{\tilde{\mathbf{X}}}

\section{Rough paths}

In this section, we discuss how to define $\int a \dif g$ for not very regular functions $a,g$.
The intended application is that $g$ is a sample path of a martingale, which only has bounded $r$-variation for $r>2$, so that we cannot use Stieltjes integration.

The function $a$ is assumed to have the same regularity as $g$, because we want the theory to be suitable for solving equations like $f(t) = \int_{0}^{t} a(s,f(s)) \dif g(s)$.

\subsection{Young integration with Hölder functions}
We begin with a criterion for convergence of Riemann-like sums.
Let $\Delta := \Set{ (s,t) \given 0 \leq s \leq t \leq T }$.
For any $\Xi : \Delta \to \R$ and a partition $0 = \pi_{0} < \dotsc < \pi_{J} = T$, write
\[
\calI^{\pi} \Xi_{0,T} := \sum_{j=1}^{J} \Xi_{\pi_{j-1},\pi_{j}}.
\]
We will discuss sufficient conditions for the convergence of these sums along the directed (by set inclusion) set of partitions.

For a mapping $\Xi : \Delta \to E$ ($E$ a vector space), we write
\[
\delta\Xi_{s,u,t}:=\Xi_{s,t}-\Xi_{s,u}-\Xi_{u,t}.
\]
A \emph{control} is a function $\omega : \Delta \to [0,\infty)$ that is superadditive in the sense that
\[
\omega(s,t) + \omega(t,u) \leq \omega(s,u)
\]
for all $s \leq t\leq u$.
This implies in particular that $\omega(t,t)=0$ for all $t$.

\begin{lemma}
\label{lem:sewing-bound}
Let $\omega$ be a control and $\theta>1$.
Let $(E,\abs{\cdot})$ be a Banach space and $\Xi : \Delta \to E$ such that, for every $0 \leq s \leq t \leq u \leq T$, we have
\begin{equation}
\label{eq:delta-control}
\abs{\delta\Xi_{s,t,u}} \leq \omega(s,u)^{\theta}.
\end{equation}
Then, for every partition $0 = \pi_{0} \leq \dotsc \leq \pi_{J} = T$, we have
\[
\abs{\Xi_{0,T} - \calI^{\pi}\Xi_{0,T}} \lesssim (\theta-1)^{-1} \omega(0,T)^{\theta}.
\]
\end{lemma}
\begin{proof}
By induction on the partition size $J$, we will show the bound
\[
\abs{\Xi_{0,T} - \calI^{\pi}\Xi_{0,T}} \leq \sum_{k=1}^{J-1} (2/k)^{\theta} \omega(0,T)^{\theta}.
\]
For $J=1$, we have $\calI^{\pi} \Xi_{0,T} = \Xi_{0,T}$, which serves as the induction base.

Suppose that the claim is known for all partitions of size $J$ and let $\pi$ be a partition of size $J+1$.
By superadditivity of $\omega$, we have
\[
\sum_{k=0}^{J-1} \omega(\pi_{k},\pi_{k+2})
=
\sum_{k\leq J-1 \text{ even}} \omega(\pi_{k},\pi_{k+2})
+
\sum_{k\leq J-1 \text{ odd}} \omega(\pi_{k},\pi_{k+2})
\leq
2 \omega(0,T).
\]
Hence, there exists $k$ with $\omega(\pi_{k},\pi_{k+2}) \leq 2 \omega(0,T)/J$.
Let $\pi' := \pi \setminus \Set{\pi_{k+1}}$.
Then,
\[
\abs{\calI^{\pi}\Xi_{0,T} - \calI^{\pi'}\Xi_{0,T}}
=
\abs{\delta\Xi_{\pi_{k},\pi_{k+1},\pi_{k+2}}}
\leq
\omega(\pi_{k},\pi_{k+2})^{\theta}
\leq (2/J)^{\theta} \omega(0,T).
\]
Applying the inductive hypothesis to $\pi'$, we obtain the claim.
\end{proof}

\begin{proposition}[Sewing]
\label{prop:sewing-vanishing-control}
Let $\omega$ be a control such that
\begin{equation}
\label{eq:vanishing-control}
\lim_{\pi} \sup_{j} \omega(\pi_{j},\pi_{j+1}) = 0.
\end{equation}
Let $\Xi : \Delta \to E$ be such that \eqref{eq:delta-control} holds.
Then, the limit
\[
\calI \Xi_{0,T} := \lim_{\pi} \calI^{\pi} \Xi_{0,T}
\]
exists and satisfies
\[
\abs{ \calI \Xi_{0,T} - \Xi_{0,T} } 
\lesssim (\theta-1)^{-1} \omega(0,T)^{\theta}.
\]
\end{proposition}
\begin{proof}
For any partitions $\pi \subseteq \pi'$, by Lemma~\ref{lem:sewing-bound}, we have
\begin{align*}
\abs{\calI^{\pi}\Xi_{s,t} - \calI^{\pi'}\Xi_{s,t}}
&\leq
\sum_{j} \abs{\Xi_{\pi_{j},\pi_{j+1}} - \calI^{\pi'}\Xi_{\pi_{j},\pi_{j+1}}}
\\ &\lesssim
\sum_{j} \omega(\pi_{j},\pi_{j+1})^{\theta}
\\ &\lesssim
\sup_{j} \omega(\pi_{j},\pi_{j+1})^{\theta-1} \sum_{j} \omega(\pi_{j},\pi_{j+1})
\\ &\lesssim
\omega(s,t) \sup_{j} \omega(\pi_{j},\pi_{j+1})^{\theta-1}.
\end{align*}
The hypothesis \eqref{eq:vanishing-control} implies that $\pi \mapsto \calI^{\pi}\Xi_{s,t}$ is a Cauchy net.
\end{proof}

\begin{example}[Young integral]
\label{example:Young}
Let $a,g : [0,T] \to \R$ be $\alpha$-Hölder functions and
\[
\Xi_{s,t} := a_{s} (g_{t}-g_{s}).
\]
Then,
\begin{align*}
\delta \Xi_{s,t,u}
&= \Xi_{s,u} - \Xi_{s,t} - \Xi_{t,u}
\\ &= a_{s} (g_{u}-g_{s}) - a_{s} (g_{t}-g_{s}) - a_{t} (g_{u}-g_{t})
\\ &= a_{s} (g_{u}-g_{t}) + a_{s} (g_{t}-g_{s}) - a_{s} (g_{t}-g_{s}) - a_{t} (g_{u}-g_{t})
\\ &= (a_{s}-a_{t}) (g_{u}-g_{t}).
\end{align*}
In particular, $\Xi_{s,t,u} \lesssim \abs{s-t}^{\alpha} \abs{u-t}^{\alpha} \leq \abs{u-s}^{2\alpha}$.
Hence, provided that $2\alpha>1$, the hypothesis of Proposition~\ref{prop:sewing-vanishing-control} holds with $\theta=2\alpha$ and $\omega(s,t) \sim \abs{t-s}$.
\end{example}

\subsection{Young integration with $V^{r}$ functions}

If the function $g : [0,T] \to \R$ has bounded $r$-variation, then
\[
\omega_{g,r}(s,t)
:= (V^{r}g_{s,t})^{r}
= \sup_{s \leq \pi_{0} \leq \dotsb \leq \pi_{J} \leq t} \sum_{j=1}^{J} \abs{g_{\pi_{j}}-g_{\pi_{j-1}}}^{r}
\]
is a control.
If both $a,g \in V^{r}$, then, in the situation of Example~\ref{example:Young}, we have
\begin{equation}
\label{eq:Yound-delta-control}
\abs{\delta \Xi_{s,t,u}}
\leq
\omega_{a,r}(s,t)^{1/r} \omega_{g,r}(t,u)^{1/r}.
\end{equation}
This implies in particular \eqref{eq:delta-control} with $\omega=\omega_{a,r}+\omega_{g,r}$ and $\theta=2/r$.
However, \eqref{eq:vanishing-control} only holds for this control if $a,g$ are also continuous.
In order to integrate functions with jumps, we will show a version of the sewing lemma under the condition \eqref{eq:Yound-delta-control}.

\begin{definition}
If $f$ is a function defined on a suitable interval, we write
\[
f(t+) := \lim_{t' \to t, t'>t} f(t'),
\quad
f(t-) := \lim_{t' \to t, t'<t} f(t'),
\]
provided that the respective limits exist.
\end{definition}

\begin{lemma}\label{lem:control-small-partition}
Let $\omega$ be a control.
Then, for every $\epsilon>0$, there exists a partition $\pi$ such that
\begin{equation}
\label{eq:control-small-partition}
\max_{j} \bigl( \omega(\pi_{j-1}+,\pi_{j}) \wedge \omega(\pi_{j-1},\pi_{j}-) \bigr) \leq \epsilon.
\end{equation}
\end{lemma}
\begin{proof}
We select the partition greedily starting with $\pi_{0}=0$.
It will be clear from the construction that the claim \eqref{eq:control-small-partition} holds.
If $\pi_{j}$ has been already selected and $\pi_{j}<T$, we distinguish two cases.

Case 1: if $\omega(\pi_{j},\pi_{j}+) < \epsilon$, then we let
\[
\pi_{j+1} := \sup \Set{t>\pi_{j} \given \omega(\pi_{j},t) < \epsilon}.
\]
Then $\omega(\pi_{j},\pi_{j+1}-) \leq \epsilon$.

Case 2: if $\omega(\pi_{j},\pi_{j}+) \geq \epsilon$, then we choose any $\pi_{j+1} \in (\pi_{j},T]$ with $\omega(\pi_{j}+,\pi_{j+1}) \leq \epsilon$.

To see that such $\pi_{j+1}$ exists, suppose for a contradiction that, for every $t \in (\pi_{j},T]$, we have $\omega(\pi_{j}+,t) > \epsilon$.
Choose recursively $t_{0}=T$ and, given $t_{k}$, let $t_{k+1} \in (\pi_{j},t_{k})$ be such that $\omega(t_{k+1},t_{k}) > \epsilon$.
By superadditivity, we obtain
\[
\omega(\pi_{j},T) \geq \sum_{k} \omega(t_{k+1},t_{k}) = +\infty,
\]
a contradiction.

The selection of $\pi_{j}$'s ends after finitely many steps, because otherwise we would have $\omega(\pi_{j},\pi_{j+1}+) \geq \epsilon$ for every $j$, which in turn implies $\omega(\pi_{j},\pi_{j+2}) \geq \epsilon$, and summing over even $j$ we obtain a contradiction with the superadditivity of $\omega$.
\end{proof}

\begin{theorem}[Sewing with jumps]
\label{thm:sewing-jumps}
Let $\omega_{1,n}, \omega_{2,n}$ be controls and $\alpha_{1,n},\alpha_{2,n} \geq 0$ with $\alpha_{1,n}+\alpha_{2,n}>1$ for all $n \in \Set{1,\dotsc,N}$.

Let $\Xi : \Delta \to E$.
Assume
\[
\abs{\delta\Xi_{s,u,t}} \leq \sum_{n=1}^N \omega_{1,n}^{\alpha_{1,n}}(s,u) \omega_{2,n}^{\alpha_{2,n}}(u,t),
\]
Then the following net limit exists,
\[
\calI\Xi_{0,T}:=\lim_{\pi} \calI^\pi \Xi_{0,T},
\]
and one has the following estimate:
\[
\abs{ \calI\Xi_{0,T}-\Xi_{0,T}}
\lesssim
\sum_{n=1}^N \omega_{1,n}^{\alpha_{1,n}}(0,T-) \omega_{2,n}^{\alpha_{2,n}}(0+,T),
\]
with $C$ depending only on $\min_n \alpha_{1,n}+\alpha_{2,n}$.
\end{theorem}
\begin{proof}
Let
\[
\theta := \min_n \alpha_{1,n}+\alpha_{2,n} > 1
\]
and
\[
\omega(s,t) := \sum_{n=1}^N \omega_{1,n}^{\alpha_{1,n}/\theta}(s,t-) \omega_{2,n}^{\alpha_{2,n}/\theta}(s+,t).
\]
Note that the functions $(s,t) \mapsto \omega_{1,n}(s,t-)$ and $(s,t) \mapsto \omega_{2,n}(s+,t)$ are also controls.
The function $\omega$ is a control, because, by Hölder's inequality,
\begin{multline*}
\omega_{1,n}^{\alpha_{1,n}/\theta}(s,u) \omega_{2,n}^{\alpha_{2,n}/\theta}(s,u)
\geq
(\omega_{1,n}(s,t) + \omega_{1,n}(t,u))^{\alpha_{1,n}/\theta} (\omega_{2,n}(s,t)+\omega_{2,n}(t,u))^{\alpha_{2,n}/\theta}
\\ \geq
\omega_{1,n}(s,t)^{\alpha_{1,n}/\theta} \omega_{2,n}(s,t)^{\alpha_{2,n}/\theta} + \omega_{1,n}(t,u)^{\alpha_{1,n}/\theta} \omega_{2,n}(t,u)^{\alpha_{2,n}/\theta}.
\end{multline*}
It follows from Lemma~\ref{lem:control-small-partition} applied to the control $\sum_{n} (\omega_{1,n} + \omega_{2,n})$ that \eqref{eq:vanishing-control} holds.
Since $\abs{\delta\Xi_{s,t,u}} \leq \omega(s,u)^{\theta}$, we can apply Proposition~\ref{prop:sewing-vanishing-control}.
\end{proof}

Theorem~\ref{thm:sewing-jumps} allows us to integrate under the hypothesis \eqref{eq:Yound-delta-control}, but only for $r<2$.
By the law of the iterated logarithm, paths of the Brownian motion are not in $V^{r}$ for any $r\leq 2$.

There are two options to refine the above reasoning in order to obtain an integration theory suitable for martingales.
\begin{itemize}
\item The classical stochastic integration uses orthogonality between martingale increments.
We will return to that.
\item Rough integration theory replaces Riemann--Stieltjes sums by higher order quadrature schemes.
Higher order quadrature improves the order of convergence (under suitable regularity assumptions).
In the rough setting, it improves from having no convergence to having convergence.
\end{itemize}
Another motivation for us to look at the rough integration theory is that it informs us about the estimates that we will want when we return to stochastic integration.

\neuevorl{2021-11-23}

\subsection{Rough paths}
Throughout this section, we fix some $r \in [2,3)$.

An $r$-rough path (with values in $\R^{d}$) consists of
\[
X : [0,T] \to \R^{d},
\quad
\bbX : \Delta \to \R^{d \times d}
\]
such that $X \in V^{r}$, $\bbX \in V^{r/2}$, where
\[
V^{r}\bbX_{s,t} = \sup_{s\leq u_{0} \leq \dotsb \leq u_{J} \leq t} \Bigl( \sum_{j} \abs{\bbX_{u_{j},u_{j+1}}}^{r} \Bigr)^{1/r}.
\]
and \emph{Chen's relation} holds for $s<t<u$:
\begin{equation}
\bbX_{s,u} = \bbX_{s,t} + \bbX_{t,u} + (X_{t}-X_{s}) \otimes (X_{u}-X_{t}).
\end{equation}
\begin{remark}
If $X$ is sufficiently regular for the integral to be defined, one can take
\[
\bbX_{s,t} := \int_{s}^{t} (X_{u-}-X_{s}) \otimes \dif X_{u}.
\]
\end{remark}

Next, we introduce the spaces in which one can state and solve differential equations driven by rough paths.

\begin{definition}[Controlled path]
Let $\bfX=(X,\bbX)$ be an $r$-rough path in $\R^{d}$.
An \emph{$\bfX$-controlled $r$-rough path} in a vector space $E$ consists of $Y : [0,T] \to E$, $Y' : [0,T] \to L(\R^{d}, E)$ such that $Y' \in V^{r}$ and
\[
R^{\bfY,\bfX}_{s,t} := \delta Y_{s,t} - Y'_{s} \delta X_{s,t} \in V^{r/2}.
\]
\end{definition}

Although $R^{\bfY,\bfX}$ depends on all of $Y,Y',X$, it is commonly abbreviated to $R^{Y}$, since the other dependencies are usually unambiguous.

It may be helpful to think about the scalar case $d=d'=1$, and we will use the scalar notation.
All arguments will be, however, formulated in such a way that they work in the vector-valued case, which is important e.g.\ because one may wish to incorporate time as an additional coordinate in $\bfX$.

\begin{example}
$(X,1)$ is an $X$-controlled path.
\end{example}

\begin{lemma}[Rough integral]
\label{lem:rough-integral}
Let $\bfX=(X,\bbX)$ be an $r$-rough path and $\bfY=(Y,Y')$ an $\bfX$-controlled path in $L(\R^{d},\R^{d'})$.
Let $\Xi_{s,t} := Y_{s} \delta X_{s,t} + Y_{s}' \bbX_{s,t}$.
Then,
\[
\int_{0}^{T} \bfY \dif \bfX
:= \lim_{\pi} \calI^{\pi} \Xi_{0,T}
\]
exists and satisfies
\[
\abs[\Big]{ \int_{s}^{t} \bfY \dif \bfX - \Xi_{s,t}}
\lesssim
V^{r/2} R_{s,t-} V^{r} X_{s+,t} + V^{r} Y'_{s,t-} V^{r/2} \bbX_{s+,t}.
\]
\end{lemma}
\begin{proof}
\begin{equation}
\label{eq:delta-Xi-2-rough-integral}
\begin{split}
\delta \Xi_{s,t,u}
&=
Y_{s} \delta X_{s,u} + Y_{s}' \bbX_{s,u}
- (Y_{s} \delta X_{s,t} + Y_{s}' \bbX_{s,t}) - (Y_{t} \delta X_{t,u} + Y_{t}' \bbX_{t,u})
\\ &=
Y_{s} \delta X_{t,u} + Y_{s}' (\bbX_{s,u} - \bbX_{s,t}) - (Y_{t} \delta X_{t,u} + Y_{t}' \bbX_{t,u})
\\ &=
(Y_{s}-Y_{t}) \delta X_{t,u} + Y_{s}' (\bbX_{t,u} + \delta X_{s,t} \otimes \delta X_{t,u}) - Y_{t}' \bbX_{t,u}
\\ &=
(Y_{s}-Y_{t}) \delta X_{t,u} + (Y_{s}'-Y_{t}') \bbX_{t,u} + Y_{s}' (\delta X_{s,t} \otimes \delta X_{t,u})
\\ &=
-R^{\bfY,\bfX}_{s,t} \delta X_{t,u} + (Y_{s}'-Y_{t}') \bbX_{t,u}.
\end{split}
\end{equation}
Hence, we can apply Theorem~\ref{thm:sewing-jumps}.
\end{proof}

\begin{example}
$\int_{0}^{T} (X-X_{0},1) \otimes \dif \bfX = \bbX_{0,T}$.
Indeed, in this case $\delta\Xi=0$ in \eqref{eq:delta-Xi-2-rough-integral}.
\end{example}

This gives even a $V^{r/3}$ approximation of the integral.
However, we drop the second order term and only use a $V^{r/2}$ approximation, since we want to run the iteration in the simplest possible space for the given regularity.

\subsubsection{Rough differential equations}
We consider an initial value problem that is informally described by $Y_{0}=y_{0}$, $\dif Y = \phi(Y) \dif X$.
A rigorous formulation is the following: given a rough path $\bfX$ in $\R^{d}$ and a function $\phi : \R^{d'} \to L(\R^{d},\R^{d'})$, we are looking for an $\bfX$-controlled path $\bfY$ such that
\begin{equation}
\label{eq:RDE:imprecise}
Y_{T} = y_{0} + \int_{0}^{T} \phi(\bfY) \bfX.
\end{equation}
We still have to define what we mean by $\phi(\bfY)$.

We will find that this problem has a unique solution if $\phi \in C^{2,1}_{b}$, the space of twice differentiable functions with $\phi,D\phi,D^{2}\phi$ bounded and $D^{2}\phi$ Lipschitz.

The proof proceeds by a fixed point argument.
First, we need to identify a metric space in which this argument will be run.
We abbreviate $\norm{X}_r := V^r X_{0,T}$.

\subsubsection{Operations on controlled paths}

We have deliberately omitted the condition $Y\in V^{r}$ from the definition of the controlled path, because it is implicit in the other conditions.

\begin{lemma}[Implicit bound]
\label{lem:controlled-implicit}
Let $\bfX$ be a rough path and $\bfY$ an $\bfX$-controlled path.
Then
\[
\norm{Y}_{r} \leq \norm{Y'}_{\sup} \norm{X}_{r} + \norm{R^{Y}}_{r/2}.
\]
\end{lemma}

\begin{lemma}[Integration]
\label{lem:controlled-integral}
Let $\bfX$ be a rough path and $\bfY$ an $\bfX$-controlled path.
Let $Z_{t}:=\int_0^t Y \dif \bfX$ and $Z_{t}' := Y_{t}$.
Then, $\bfZ=(Z,Z')$ is an $\bfX$-controlled path, and it satisfies
\begin{equation}
\label{eq:controlled-integral-remainder}
\norm{R^{Z}}_{r/2}
\lesssim
\norm{R^{Y}}_{r/2} \norm{X}_{r} + \norm{Y'}_{r} \norm{\bbX}_{r/2} + \norm{Y'}_{\sup} \norm{\bbX}_{r/2}.
\end{equation}
\end{lemma}
\begin{proof}
\begin{align*}
R^{Z}_{s,t}
&=
\delta Z_{s,t} - Z'_{s} \delta X_{s,t}
\\ &=
\int_{0}^{t} \bfY \dif \bfX - \int_{0}^{s} \bfY \dif \bfX - Y_{s} \delta X_{s,t}
\\ &=
\int_{s}^{t} \bfY \dif \bfX - Y_{s} \delta X_{s,t} - Y'_{s} \bbX_{s,t} + Y'_{s} \bbX_{s,t}.
\end{align*}
The last term corresponds to the last term in the conclusion, and the remaining difference is estimated by Lemma~\ref{lem:rough-integral}.
\end{proof}

\begin{lemma}[Composition with a smooth function]
\label{lem:controlled-composition}
Let $\bfX$ be a rough path and $\bfY$ an $\bfX$-controlled path.
Suppose $\phi \in C^{1,1}_{b}$.
Then $\phi(\bfY) := (\phi(Y), D\phi(Y)Y')$ is an $\bfX$-controlled path, and we have
\begin{equation}
\label{eq:controlled-composition-phi'}
\norm{\phi(\bfY)'}_{r}
=
\norm{D\phi(Y)Y'}_{r}
\leq
\norm{D\phi}_{\sup} \norm{Y'}_{r} + \norm{D\phi}_{Lip} \norm{Y}_{r} \norm{Y'}_{\sup},
\end{equation}
\begin{equation}
\label{eq:controlled-composition-R}
\norm{ R^{\phi(Y)} }_{r/2}
\leq
\norm{D\phi}_{\sup} \norm{R^Y}_{r/2}
+
\frac12 \norm{D \phi}_{Lip} \norm{Y}_{r}^{2}.
\end{equation}
\end{lemma}
Interestingly, this estimate does not depend on $\bfX$.
\begin{proof}
Adding and subtracting $D\phi(Y_t)Y'_s$, we obtain
\begin{equation}
\label{eq:4}
D\phi(Y_t)Y'_t- D\phi(Y_s)Y'_s
=
D\phi(Y_t)(Y'_t-Y'_{s}) + (D\phi(Y_t)-D\phi(Y_s)) Y'_s.
\end{equation}
This implies the first estimate.

By Taylor's formula,
\begin{align*}
\abs{R^{\phi(Y)}_{s,t}}
&=
\abs{\phi(Y_t)- \phi(Y_s) -D \phi(Y_s) Y_s' X_{s,t}}
\\ &\leq
\abs{\phi(Y_t)- \phi(Y_s) -D \phi(Y_s) Y_{s,t}} + \abs{D\phi(Y_s) R^Y_{s,t}}
\\ &\leq
\frac12 \norm{D \phi}_{Lip} \abs{Y_{s,t}}^{2} + \norm{D\phi}_{\sup} \abs{R^Y_{s,t}},
\end{align*}
which implies second estimate.
\end{proof}

The above discussion allows us to constrct a space in which the equation \eqref{eq:RDE} makes sense.
\begin{lemma}[Solution space]
For any $\phi \in C^{1,1}_{b}$ and any $A \in (0,\infty)$, there exists $\epsilon>0$ such that if $\norm{X}_{r} + \norm{\bbX}_{r/2} < \epsilon$, then the set of $\bfX$-controlled paths
\begin{equation}
\label{eq:RDE-solution-space}
\calY = \calY(\bfX,A) := \Set{ \bfY \given \norm{Y'}_{r} \leq A, \norm{R^{Y}}_{r/2} \leq A^{2}, \norm{Y}_{r} \leq A/\norm{\phi}_{Lip}, \norm{Y'}_{\sup} \leq \norm{\phi}_{\sup} }
\end{equation}
is invariant under the mapping
\begin{equation}
\label{eq:RDE-iteration}
\operatorname{Step} : \bfY \mapsto (y_{0} + \int_{0}^{\cdot} \phi(\bfY) \dif \bfX, \phi(Y))
\end{equation}
for any $y_{0} \in \R^{d'}$.
\end{lemma}
\begin{proof}
Implicit constants in this proof are allowed to depend on the $C^{1,1}_{b}$ norm of $\phi$.

Suppose $\bfY \in \calY$.
Direct estimates show
\[
\norm{\phi(Y)}_{r}  \leq \norm{\phi}_{Lip} \norm{Y}_{r} \leq A,
\quad
\norm{\phi(Y)}_{\sup} \leq \norm{\phi}_{\sup}.
\]
By Lemma~\ref{lem:controlled-composition}, we have
\begin{equation}
\label{eq:3}
\norm{\phi(\bfY)'}_{r} \lesssim A,
\quad
\norm{R^{\phi(\bfY)}}_{r/2} \lesssim A^{2}.
\end{equation}
By Lemma~\ref{lem:controlled-integral}, we obtain
\[
\norm{R^{\int_{0}^{\cdot}\phi(\bfY)\dif\bfX}}_{r/2}
\lesssim
\epsilon (A^{2} + A + 1).
\]
By Lemma~\ref{lem:controlled-implicit}, this implies
\[
\norm{\int_{0}^{\cdot}\phi(\bfY)\dif\bfX}_{r}
\lesssim
\epsilon (A^{2} + A + 1).
\]
Choosing $\epsilon$ sufficiently small, we obtain the claim.
\end{proof}

\begin{remark}
Global bounds on $\phi$ can be replaced by growth conditions, and then $\epsilon$ would also depend on $y_{0}$ and these growth conditions.
There is also a ``rough Gronwall lemma'' for concatenating local solutions in such a setting.
\end{remark}

\neuevorl{2021-11-30}

\subsubsection{Contractive iteration}

In this section, we will show that the iteration \eqref{eq:RDE-iteration} is contractive with respect to a suitable metric on the space \eqref{eq:RDE-solution-space}.
This implies existence and uniqueness of solutions.

For the later purpose of proving continuous dependence of the solution on data, the estimates will involve two rough paths $\bfX,\tilde\bfX$ and paths $\bfY,\tilde\bfY$ controlled by the respective rough paths.
We abbreviate
\begin{multline*}
\Delta Y = Y - \tY,
\quad
\Delta Y' = Y' - \tY',
\quad
\Delta R^{Y} = R^{Y} - R^{\tY},
\quad
\Delta X = X - \tX,
\quad
\Delta \bbX = \bbX - \tilde{\bbX},
\\ \quad
\Delta \phi(Y) = \phi(Y) - \phi(Y'),
\quad
\Delta R^{\phi(Y)} = R^{\phi(Y)} - R^{\phi(\tY)}.
\end{multline*}

\begin{lemma}[Stability of composition]
\label{lem:controlled-composition-stable}
Suppose $\phi \in C^{2,1}_{b}$.
Let $\bfX,\tilde{\bfX}$ be rough paths, $\bfY$ a $\bfX$-controlled path, and $\tilde{\bfY}$ a $\tilde{\bfX}$-controlled path.
Then,
\begin{equation}
\label{eq:controlled-composition-Delta-phi}
\norm{\phi(Y)-\phi(\tY)}_{r}
\leq
\norm{D\phi}_{\sup} \norm{\Delta Y}_{r} + \norm{D\phi}_{Lip} \norm{\Delta Y}_{\sup} \norm{\tY}_{r},
\end{equation}
\begin{multline}
\label{eq:controlled-composition-Delta-phi'}
\norm{\phi(Y)'-\phi(\tY)'}_{r}
\leq
\norm{D\phi}_{\sup} \norm{\Delta Y'}_{r}
+ \norm{D\phi}_{Lip} \norm{Y}_{r} \norm{\Delta Y'}_{\sup}
\\+ \norm{D\phi}_{Lip} \norm{\Delta Y}_{\sup} \norm{\tY'}_{r}
+ \norm{D^{2}\phi}_{\sup} \norm{\Delta Y'}_{r} \norm{\tY'}_{\sup}
+ \norm{D^{2}\phi}_{Lip} \norm{\Delta Y}_{\sup} \norm{\tY}_{r} \norm{\tY'}_{\sup},
\end{multline}
\begin{multline}
\label{eq:controlled-composition-Delta-R}
\norm{R^{\phi(Y)}- R^{\phi(\tY)}}_{r/2}
\leq
\norm{D\phi}_{\sup} \norm{ \Delta R^Y }_{r/2}
+ \norm{D\phi}_{Lip} \norm{\Delta Y}_{\sup} \norm{R^{\tY} }_{r/2}
\\+\frac12 \norm{D^{2}\phi}_{\sup} (\norm{Y}_r+\norm{\tY }_r) \norm{\Delta Y}_r
+\frac12 \norm{D^{2}\phi}_{Lip} \norm{\Delta Y}_{\sup} \norm{\tY}_r^2.
\end{multline}
\end{lemma}

Lemma~\ref{lem:controlled-composition-stable} almost recovers Lemma~\ref{lem:controlled-composition} upon setting $\tilde\bfY=0$.

\begin{proof}
In order to show \eqref{eq:controlled-composition-Delta-phi}, we write
\begin{multline}
\label{eq:7}
(\phi(Y_t)-\phi(Y_s)) - (\phi(\tY_t)-\phi(\tY_s))
\\=
\int_{0}^{1} D\phi(Y_{s}+r(Y_{t}-Y_{s})) (Y_{t}-Y_{s}) \dif r
-
\int_{0}^{1} D\phi(\tY_{s}+r(\tY_{t}-\tY_{s})) (\tY_{t}-\tY_{s}) \dif r
\\=
\int_{0}^{1} \Bigl(
D\phi(Y_{s}+r(Y_{t}-Y_{s})) (Y_{t}-Y_{s} - (\tY_{t}-\tY_{s}))
\\+
(D\phi(Y_{s}+r(Y_{t}-Y_{s})) - D\phi(\tY_{s}+r(\tY_{t}-\tY_{s}))) (\tY_{t}-\tY_{s}) \Bigr) \dif r,
\end{multline}
and estimate
\begin{align*}
\MoveEqLeft
\abs{D\phi(Y_{s}+r(Y_{t}-Y_{s})) - D\phi(\tY_{s}+r(\tY_{t}-\tY_{s}))}
\\ &\leq
\norm{D\phi}_{Lip} \abs{(Y_{s}+r(Y_{t}-Y_{s})) - (\tY_{s}+r(\tY_{t}-\tY_{s}))}
\\ &\leq
\norm{D\phi}_{Lip} (r\abs{Y_{t}-\tY_{t}} + (1-r) \abs{Y_{s}-\tY_{s}})
\\ &\leq
\norm{D\phi}_{Lip} \norm{\Delta Y}_{\sup}.
\end{align*}
In order to show \eqref{eq:controlled-composition-Delta-phi'}, we write
\begin{multline*}
(D\phi(Y_t)Y'_t-D\phi(\tY_t)\tY'_t) - (D\phi(Y_s)Y'_s-D\phi(\tY_s)\tY'_s)
\\ =
D\phi(Y_t) ((Y'_t-\tY'_t)-(Y'_{s}-\tY'_{s}))
+ (D\phi(Y_t)-D\phi(Y_s)) (Y'_s-\tY'_{s})
\\ + (D\phi(Y_{t})-D\phi(\tY_t))(\tY'_t-\tY'_{s})
+ ((D\phi(Y_t)-D\phi(Y_s))-(D\phi(\tY_t)-D\phi(\tY_s))) \tY'_s.
\end{multline*}
The first 3 terms contribute the first 3 terms to \eqref{eq:controlled-composition-Delta-phi'}.
The 4-fold difference in the last term can be written as in \eqref{eq:7} with $D\phi$ in place of $\phi$, which gives the last 2 terms in \eqref{eq:controlled-composition-Delta-phi'}.

Now we show \eqref{eq:controlled-composition-Delta-R}.
By inserting $D\phi(Y_s)Y_{s,t}$ and $ D\phi(\tY_s) \tY_{s,t},$ one has
\begin{equation}
\label{eq:5}
\begin{split}
R^{\phi(Y)}_{s,t}- R^{\phi(\tY)}_{s,t}
&= \phi(Y)_{s,t} - D \phi(Y_s) Y_s' X_{s,t} - (\phi(\tY)_{s,t} -D \phi(\tY_s) \tY_s' \tX_{s,t})
\\ &=
\phi(Y)_{s,t}-D\phi(Y_s)Y_{s,t} - (\phi(\tY)_{s,t}- D\phi(\tY_s) \tY_{s,t})
\\ &+ D\phi(Y_s)R^Y_{s,t}- D\phi(\tY_s)R^{\tY}_{s,t}
\end{split}
\end{equation}
By Taylor's formula, the first line in \eqref{eq:5} equals
\[
\int_0^1 (1-r) \bigl( D^2\phi(Y_s + r Y_{s,t}) Y_{s,t}^{\otimes 2} - D^2\phi(\tY_s + r \tY_{s,t}) \tY_{s,t}^{\otimes 2} \bigr) \dif r.
\]
The integrand can be written in the form
\begin{equation*}
D^2\phi(Y_s + r Y_{s,t}) (Y_{s,t}^{\otimes 2} - \tY_{s,t}^{\otimes 2})
+
(D^2\phi(Y_s + r Y_{s,t}) - D^2\phi(\tY_s + r \tY_{s,t})) \tY_{s,t}^{\otimes 2},
\end{equation*}
and this contributes the last line to \eqref{eq:controlled-composition-Delta-R}.

The second line in \eqref{eq:5} equals
\begin{align*}
D\phi(Y_s)R^Y_{s,t}- D\phi(\tY_s)R^{\tY}_{s,t}
&\lesssim
D\phi(Y_s) (R^Y_{s,t}-R^{\tY}_{s,t})
+ (D\phi(Y_s)-D\phi(\tY_s)) R^{\tY}_{s,t}
\end{align*}
which completes the proof.
\end{proof}

\begin{lemma}[Stability of rough integration]
\label{lem:controlled-integration-stable}
Let $\bfX,\tilde{\bfX}$ be rough paths, $\bfY$ a $\bfX$-controlled path, and $\tilde{\bfY}$ a $\tilde{\bfX}$-controlled path.
Define $\bfZ,\tilde{\bfZ}$ as in Lemma~\ref{lem:controlled-integral}.
Then,
\begin{multline}
\label{eq:controlled-integral-remainder-Delta}
\norm{\Delta R^{Z}}_{r/2}
\lesssim_{r}
\norm{\Delta Y'}_{\sup} \norm{\bbX}_{r/2}
+ \norm{\Delta Y'}_{r} \norm{\bbX}_{r/2}
+ \norm{\Delta R^Y}_{r/2} \norm{X}_{r}
\\ + \norm{\tilde{Y}'}_{\sup} \norm{\Delta \bbX}_{r/2}
+ \norm{\tY'}_{r} \norm{ \Delta \bbX}_{r/2}
+ \norm{R^{\tY}}_{r/2} \norm{ \Delta X}_{r}.
\end{multline}
\end{lemma}

Note that \eqref{eq:controlled-integral-remainder-Delta} with $\tilde\bfY=0$ recovers \eqref{eq:controlled-integral-remainder}.

\begin{proof}
Let $\Xi_{s,t} := Y_s X_{s,t} + Y'_s \bbX_{s,t}-(\tY_s \tX_{s,t} + \tY'_s \tilde{\bbX}_{s,t}).$
Then one has,
\begin{align*}
\abs{ R^{I_\bfX(Y)}_{s,t} - R^{I_{\tBX}(\tY)}_{s,t} }
&= \abs{\int_s^t \bfY \dif\bfX - Y_s X_{s,t} -(\int_s^t \tilde\bfY \dif\tilde\bfX - \tY_s \tX_{s,t})}
\\ &\leq
\abs{Y'_s\bbX_{s,t} - \tY'_s \tilde{\bbX}_{s,t} } + \abs{I(\Xi)_{s,t} - \Xi_{s,t} }.
\end{align*}
We estimate the first term by
\[
\abs{\Delta Y'_{s} \bbX_{s,t}} + \abs{\tilde{Y}'_{s} \Delta \bbX_{s,t}}.
\]
We estimate the last term by the sewing lemma (Theorem~\ref{thm:sewing-jumps}) with (as the calculation \eqref{eq:delta-Xi-2-rough-integral} shows)
\[
\delta(\Xi)_{\tau,u,\nu}
= R^Y_{\tau,u} \delta X_{u,\nu} + \delta Y'_{\tau,u} \bbX_{u,\nu} -(R^{\tY}_{\tau,u} \delta \tX_{u,\nu} + \delta\tY'_{\tau,u} \tilde{\bbX}_{u,\nu}).
\]
and
\begin{equation*}
\abs{\delta (\Xi)_{\tau,u,\nu}} \leq
\abs{\Delta \delta Y'_{\tau,u} \bbX_{u,\nu } } +\abs{\delta\tY'_{\tau,u} \Delta \bbX_{u,\nu}} +\abs{\Delta R^Y_{\tau,u} \delta X_{u,\nu}} + \abs{R^{\tY}_{\tau,u} \Delta \delta X_{u,\nu}}.
\end{equation*}
\end{proof}

\begin{lemma}[Stability of implicit bound]
\label{lem:controlled-implicit-stable}
\[
\norm{Y-\tY}_{r} \leq
\norm{\Delta Y'}_{\sup} \norm{X}_{r}
+ \norm{\tY'}_{\sup} \norm{\Delta X}_{r}
+ \norm{\Delta R^{Y}}_{r/2}.
\]
\end{lemma}
\begin{proof}
This follows from writing
\begin{multline*}
Y_{s,t} - \tY_{s,t}
=
(Y'_{s} X_{s,t} + R^{Y}_{s,t}) - (\tY'_{s} \tX_{s,t} + R^{\tY}_{s,t})
\\=
(Y'_{s}-\tY_{s}) X_{s,t}
+
\tY_{s} (X_{s,t} - \tX_{s,t})
+ (\Delta R^{Y}_{s,t}).
\qedhere
\end{multline*}
\end{proof}

\begin{lemma}[Contractivity of the iteration]
\label{lem:RDE-iteration-contractive}
For every $A<\infty$ and $\phi \in C^{2,1}_{b}$, there esists $\epsilon>0$ such that, for every rough path $\bfX$ with $\norm{X}_{r} < \epsilon$ and $\norm{\bbX}_{r/2} < \epsilon$,
the map \eqref{eq:RDE-iteration} is strictly contractive on the set
\begin{equation}
\label{eq:6}
\Set{Y\in\calY(\bfX,A) \given Y_{0}=y_{0}, Y'_{0}=\phi(y_{0})}
\end{equation}
with respect to the metric
\begin{equation}
\label{eq:RDE-contraction-metric}
d(\bfY,\tilde{\bfY}) = \max( \norm{\Delta R^{Y}}_{r/2} , \norm{\Delta Y'}_{r} , 2(\norm{D\phi}_{\sup}+A\norm{D\phi}_{Lip}) \norm{\Delta Y}_{r} ).
\end{equation}
\end{lemma}
\begin{proof}
The implicit constants here are allowed to depend on $A$ in the definition of $\calY$.

Suppose $\bfY,\tilde{\bfY} \in \calY(\bfX,A)$ with $d(\bfY,\tilde{\bfY}) = \alpha$ and $Y_{0}=\tY_{0}=y_{0}$.
Then,
\[
\norm{\Delta Y}_{\sup} \leq \norm{\Delta Y}_{r} \lesssim \alpha.
\]
By Lemma~\ref{lem:controlled-composition-stable}, we obtain
\[
\norm{\phi(Y)-\phi(\tY)}_{r}
\leq \norm{D\phi}_{\sup} \norm{\Delta Y}_{r} + \norm{D\phi}_{Lip} \norm{\Delta Y}_{\sup} \norm{\tY}_{r}
\leq \alpha/2.
\]
\begin{multline*}
\norm{\phi(Y)'-\phi(\tY)'}_{r}
\lesssim_{\phi}
\norm{\Delta Y'}_{r}
+ \norm{Y}_{r} \norm{\Delta Y'}_{\sup}
\\+ \norm{\Delta Y}_{\sup} \norm{\tY'}_{r}
+ \norm{\Delta Y'}_{r} \norm{\tY'}_{\sup}
+ \norm{\Delta Y}_{\sup} \norm{\tY}_{r} \norm{\tY'}_{\sup}
\lesssim
\alpha (1 + A),
\end{multline*}
\begin{multline*}
\norm{R^{\phi(Y)}- R^{\phi(\tY)}}_{r/2}
\lesssim
\norm{ \Delta R^Y }_{r/2}
+ \norm{\Delta Y}_{\sup} \norm{R^{\tY} }_{r/2}
\\+ (\norm{Y}_r+\norm{\tY }_r) \norm{\Delta Y}_r
+ \norm{\Delta Y}_{\sup} \norm{\tY}_r^2
\lesssim \alpha (1+A^{2}).
\end{multline*}
Let $\bfZ = \operatorname{Step}(\bfY)$, $\tilde\bfZ = \operatorname{Step}(\tilde\bfY)$.
Then
\[
\norm{\Delta Z'}_{r} = \norm{\phi(Y)-\phi(\tY)}_{r}
\leq \alpha/2.
\]
By Lemma~\ref{lem:controlled-integration-stable},
\[
\norm{\Delta R^{Z}}_{r/2}
\lesssim
\epsilon ( \norm{\Delta \phi(Y)'}_{\sup}
+ \norm{\Delta \phi(Y)' }_{r}
+ \norm{\Delta R^{\phi(Y)}}_{r/2} )
\lesssim
\alpha \epsilon.
\]
By Lemma~\ref{lem:controlled-implicit-stable}, this implies
\[
\norm{\Delta Z}_{r}
\lesssim
\norm{\Delta Z'}_{\sup} \norm{X}_{r}
+ \norm{\Delta R^{Z}}_{r/2}
\leq
\norm{\Delta Z'}_{r} \epsilon
+ C \epsilon \alpha
\lesssim
\epsilon \alpha.
\]
Choosing $\epsilon$ small enough, we obtain the claim.
\end{proof}

\begin{theorem}[Existence of solutions]
\label{thm:RDE-solution-local-uniqueness}
For every $\phi \in C^{2,1}_{b}$, there exist $A<\infty$ and $\epsilon>0$ such that, for every rough path $\bfX$ with $\norm{X}_{r} < \epsilon$ and $\norm{\bbX}_{r/2} < \epsilon$ and every $y_{0}$, there exists a unique $\bfX$-controlled path $\bfY \in \calY(\bfX,A)$ such that, for every $t\in [0,T]$, we have
\begin{equation}
\label{eq:RDE}
\bfY_{t} = y_{0} + \int_{0}^{t} \phi(\bfY) \dif\bfX,
\quad
Y'_{t} = \phi(Y_{t}).
\end{equation}
\end{theorem}
\begin{proof}
Let $A,\epsilon$ be as in Lemma~\ref{lem:RDE-iteration-contractive}.
Define a sequence of controlled paths by
\[
(\bfY_{0})_{t}=y_{0}, \quad (\bfY_{0})'_{t}=0,
\quad
(\bfY_{j+1}) = \operatorname{Step}(\bfY_{j}).
\]
For $j\geq 1$, the paths $(\bfY_{j})$ are elements of \eqref{eq:6}.
It follows from Lemma~\ref{lem:RDE-iteration-contractive} that this sequence is Cauchy with respect to the metric \eqref{eq:RDE-contraction-metric}, and its limit solves the RDE \eqref{eq:RDE}.
Uniqueness of the solution follows from strict contractivity of \eqref{eq:RDE-iteration}.
\end{proof}

One can also show that the solution $\bfY$ in Theorem~\ref{thm:RDE-solution-local-uniqueness} is unique among all $\bfX$-controlled paths (not only those in $\calY(\bfX,A)$).
To this end, we note that, given any two $\bfX$-controlled solutions, they are in $\calY(\bfX,A)$ for some large $A$.
One then has to subdivide the time interval into smaller intervals, on each of which $\bfX$ has sufficiently small variation norm to apply Lemma~\ref{lem:RDE-iteration-contractive}.
Large jumps of $\bfX$ have to be handled separately.
We omit the tedious details.

\neuevorl{2021-12-07}

\subsubsection{Lipschitz dependence on data}

\begin{lemma}[Local Lipschitz dependence of RDE solution on the data]
For every $\phi\in C^{2,1}_{b}$ and every $A<\infty$, there exists $\epsilon>0$ such that if $\bfX,\tilde\bfX$ are rough paths with norms $<\epsilon$,  $\bfY \in \calY(\bfX,A)$ is a solution of \eqref{eq:RDE} with initial datum $y_{0}$, and $\tilde\bfY \in \calY(\tilde\bfX,A)$ is a solution of \eqref{eq:RDE} with initial datum $\tilde{y}_{0}$ such that $\abs{y_{0}-\tilde{y}_{0}} \leq \epsilon$ and rough path $\tilde\bfX$, then
\[
\max(\norm{\Delta R^{Y}}_{r/2},\norm{\Delta Y'}_{r},\norm{\Delta Y}_{r})
\lesssim_{r,\phi,A}
\max(\norm{X-\tX}_{r},\norm{\bbX-\tilde{\bbX}}_{r/2},\abs{y_{0}-\tilde{y}_{0}}).
\]
\end{lemma}
\begin{proof}
Let $\beta$ be the RHS of the conclusion and
\[
\alpha := \max(\norm{\Delta R^{Y}}_{r/2},\norm{\Delta Y'}_{r}, 2 (\norm{D\phi}_{\sup} + A \norm{D\phi}_{Lip}+1) \norm{\Delta Y}_{r}).
\]
We may assume $\alpha > \beta$, since otherwise the conclusion already holds.
In this case, we have
\[
\norm{\Delta Y}_{\sup}
\leq \abs{y_{0}-\tilde{y}_{0}} + \norm{\Delta Y}_{r}
\lesssim \alpha,
\quad
\norm{\Delta Y'}_{\sup}
\leq \norm{\phi}_{Lip} \norm{\Delta Y}_{\sup}
\lesssim \alpha.
\]
By Lemma~\ref{lem:controlled-composition-stable}, we obtain
\[
\norm{\phi(Y)'-\phi(\tY)'}_{r} \lesssim \alpha (1+A),
\quad
\norm{\Delta R^{\phi}}_{r/2} \lesssim \alpha (1+A^{2}),
\]
A simple bound is
\[
\norm{\phi(\tY)}_{r} \leq \norm{\phi}_{Lip} \norm{\tY}_{r} \leq A.
\]
As in \eqref{eq:3}, we have
\[
\norm{R^{\phi(\tilde\bfY)}}_{r/2} \lesssim A^{2}.
\]
Inserting all these bounds into Lemma~\ref{lem:controlled-integration-stable}, we obtain
\begin{multline*}
\norm{\Delta R^{Y}}_{r/2}
\lesssim
\epsilon (\norm{\Delta \phi(Y)'}_{\sup} + \norm{\Delta \phi(Y)'}_{r} + \norm{\Delta R^{\phi(Y)}}_{r/2})
\\+ \beta ( \norm{\phi(\tilde{Y})'}_{\sup} + \norm{\phi(\tY)}_{r} + \norm{R^{\phi(\tY)}}_{r/2} )
\\ \lesssim
\epsilon \alpha (1+A^{2})
+ \beta ( 1+A^{2} ).
\end{multline*}
By Lemma~\ref{lem:controlled-composition-stable}, we have
\begin{multline*}
\norm{\Delta Y'}_{r}
=
\norm{\phi(Y)-\phi(\tY)}_{r}
\leq
\norm{D\phi}_{\sup} \norm{\Delta Y}_{r} + \norm{D\phi}_{Lip} \norm{\Delta Y}_{\sup} \norm{\tY}_{r}
\\ \leq
(\norm{D\phi}_{\sup} + A \norm{D\phi}_{Lip}) \norm{\Delta Y}_{r}
+ A \norm{D\phi}_{Lip} \abs{y_{0}-\tilde{y}_{0}}
\leq \alpha/2 + C\beta.
\end{multline*}
Moreover, by Lemma~\ref{lem:controlled-implicit-stable}, we have
\begin{align*}
\norm{\Delta Y}_{r}
&\leq
\norm{\Delta Y'}_{\sup} \norm{X}_{r}
+ \norm{\tY'}_{\sup} \norm{\Delta X}_{r}
+ \norm{\Delta R^{Y}}_{r/2}.
\\ &\leq
C \alpha \epsilon + A \beta + C (\epsilon\alpha+\beta).
\end{align*}
Inserting these bounds into the definition of $\alpha$, we obtain
\[
\alpha \leq
\max(C(\epsilon\alpha+\beta), \alpha/2+C\beta, C(\epsilon\alpha+\beta)),
\]
where $C$ depends on $r,\phi,A$.
If $C\epsilon\leq 1/2$, then this implies $\alpha \leq C \beta$.
\end{proof}

\begin{remark}
Rough paths were introduced in \cite{MR1654527} and controlled paths in \cite{MR2091358}.
For a long time, the theory concentrated on Hölder continuous paths; a good exposition of this case is in the book \cite{MR4174393}.
The treatment of $V^{r}$ paths is adapted from \cite{MR3770049}.
\end{remark}


\neuevorl{2021-12-14}

\section{Martingale transforms}
The main result of this section, Theorem~\ref{thm:vv-pprod-delta-f}, is a bound for discrete time versions of the It\^{o} integral.

We denote $\ell^{p}$ norms by
\[
\ell^{p}_{k} a_{k} := ( \sum_{k \in \N} \abs{a_{k}}^{p} )^{1/p}.
\]
In order to simplify notation, we only consider martingales $g$ with $g_{0}=0$.

\subsection{Davis decomposition}
We will use the following $L^{q}$ bound for the Davis decomposition (constructed in Lemma~\ref{lem:Davis-decomposition}).
\begin{lemma}[Davis decomposition in $L^{q}$]
\label{lem:Davis-decomposition:Lq}
For every martingale $(f_n)$ with values in a Banach space $X$, there exists a decomposition $f = f^{\pred} + f^{\bv}$ as a sum of two martingales adapted to the same filtration with $f^{\pred}_{0}=0$ such that the differences of $f^{\pred}$ have predictable majorants:
\begin{equation}
\label{eq:davis-predictable}
X d f^{\pred}_{n} \leq 2 MX df_{n-1}
\end{equation}
and $f^{\bv}$ has bounded variation, in an integral sense for every $q\in [1,\infty)$:
\begin{equation}
\label{eq:davis:bv}
L^{q} \sum_{k\leq n} X d f^{\bv}_{k}
\leq (q+1) L^{q} M X df_{n}.
\end{equation}
\end{lemma}
\begin{proof}
Recall from the construction in Lemma~\ref{lem:Davis-decomposition} that
\[
d f^{\bv}_n
=
d\tilde{h}_n - \E_{n-1}(d\tilde{h}_n),
\]
where $\tilde{h}_{n}$ is a process such that
\[
X d\tilde{h}_n
=
M X df_{n} - M X df_{n-1}.
\]
Therefore,
\[
L^{q} \sum_{k\leq n} X d f^{\bv}_{k}
\leq
L^{q} \sum_{k\leq n} X d \tilde{h}_{k}
+
L^{q} \sum_{k\leq n} \E_{k-1} X d \tilde{h}_{k}.
\]
Using Lemma~\ref{lem:sum-of-Ek} in the second summand, we obtain the claim.
\end{proof}

\begin{lemma}
\label{lem:davis-dec:Spred}
Let $1 \leq q < \infty$, $X$ be a Banach function space, elements of which are $\R$-valued maps $x(\cdot)$, and $(f_{n})$ a martingale with values in $X$.
Then for $f^{\pred}$ given by Lemma~\ref{lem:Davis-decomposition:Lq} we have
\[
\norm{ \norm{Sf^{\pred}}_{X} }_{L^{q}} \leq (q+2) \norm{ \norm{Sf}_{X} }_{L^{q}},
\]
where the square function is given by
\[
\norm{Sf}_{X} := \norm{ \ell^2_{n} ( df_n (\cdot)) }_X
\]
\end{lemma}
\begin{remark} We will apply this with $X=\ell^r$, i.e. $r$-summable series, viewed as maps from $\N \to \R$, with the usual Banach structure.
\end{remark}
\begin{proof}
Using \eqref{eq:davis:bv} we estimate
\begin{align*}
\norm{ \norm{ S f^{\pred} }_{X} }_{L^{q}}
&\leq
\norm{ \norm{ S f }_{X} }_{L^{q}}
+
\norm{ \norm{ S f^{\bv} }_{X} }_{L^{q}}
\\ &\leq
\norm{ \norm{ S f }_{X} }_{L^{q}}
+
\norm{ \norm{ \sum_{n} \abs{d_{n} f^{\bv}} }_{X} }_{L^{q}}
\\ &\leq
\norm{ \norm{ S f }_{X} }_{L^{q}}
+
\norm{ \sum_{n} \norm{ d_{n} f^{\bv} }_{X} }_{L^{q}}
\\ &\leq
\norm{ \norm{ S f }_{X} }_{L^{q}}
+
(q+1) \norm{ \sup_{n} \norm{ d_{n} f }_{X} }_{L^{q}}
\\ &\lesssim
\norm{ \norm{ S f }_{X} }_{L^{q}}.
\qedhere
\end{align*}
\end{proof}

\subsection{Vector-valued maximal paraproduct estimate}
We call a process $(F_{s,t})_{s \leq t}$ depending on two time variables \emph{adapted} if $F_{s,t}$ is $\calF_{t}$-measurable for every $s \leq t$.

For an adapted process $(F_{s,t})$ and a martingale $(g_{n})$, we define
\begin{equation}
\label{eq:paraprod}
\Pi(F,g)_{s,t} := \sum_{s < j \leq t} F_{s,j-1} dg_{j}
\end{equation}
Note that $\Pi(F,g)_{s,\cdot}$ only depends on $(F_{s,\cdot})$.

\begin{proposition}
\label{prop:vv-pprod}
Let $0 < q,q_{1} \leq \infty$, $1 \leq q_{0},r,r_{0} < \infty$, $1 \leq r_{1} \leq \infty$.
Assume $1/q = 1/q_{0} + 1/q_{1}$ and $1/r = 1/r_{0} + 1/r_{1}$.
Then, for any martingales $(g^{(k)}_{n})_{n}$, any adapted sequences $(F^{(k)}_{s,t})_{s\leq t}$, and any stopping times $\tau_{k}' \leq \tau_{k}$ with $k\in\Z$, we have
\begin{equation}
\label{eq:vv-pprod}
\norm[\big]{ \ell^{r}_{k} \sup_{\tau_{k}' \leq t \leq \tau_{k}} \abs{\Pi(F^{(k)},g^{(k)})_{\tau_{k}',t}} }_{q}
\leq C_{q_{0},q_{1},r_{0},r_{1}}
\norm[\big]{ \ell^{r_{1}}_{k} \sup_{\tau_{k}' \leq t < \tau_{k}} \abs{F^{(k)}_{\tau_{k}',t}} }_{q_{1}}
\norm{ \ell^{r_{0}}_{k} Sg^{(k)}_{\tau_{k}',\tau_{k}} }_{q_{0}},
\end{equation}
where $Sg_{s,t} := (\delta (S g)^{2})^{1/2}_{s,t} = \bigl( \sum_{j=s+1}^{t} \abs{dg_{j}}^{2} \bigr)^{1/2}$.
\end{proposition}

\begin{proof}[Proof of Proposition~\ref{prop:vv-pprod}]
We may replace each $g^{(k)}$ by the martingale
\begin{equation}
\label{eq:stop-g}
\tilde{g}^{(k)}_{n}
:=
g^{(k)}_{n \wedge \tau_{k}} - g^{(k)}_{n \wedge \tau_{k}'}
\end{equation}
without changing the value of either side of \eqref{eq:vv-pprod}.

Consider first $q \geq 1$.
For each $k$, the sequence
\[
h^{(k)}_{t}
:=
\begin{cases}
0, & t < \tau_{k}',\\
\Pi(F^{(k)},g^{(k)})_{\tau_{k}',t}, & t \geq \tau_{k}',
\end{cases}
\]
is a martingale.
We may also assume $F_{\tau_{k}',t} = 0$ if $t \not\in [\tau_{k}',\tau_{k})$.
By the $\ell^{r}$ valued BDG inequality (Corollary~\ref{cor:lr-BDG}), we can estimate
\begin{align*}
LHS~\eqref{eq:vv-pprod}
& \lesssim
\norm[\big]{ \ell^{r}_{k} \abs{S h^{(k)}} }_{q}
\\ &=
\norm[\big]{ \ell^{r}_{k} \ell^{2}_{j} \abs{F_{\tau_{k}',j-1}^{(k)} dg^{(k)}_{j}} }_{q}
\\ &\leq
\norm[\big]{ \ell^{r}_{k} MF^{(k)} \ell^{2}_{j} \abs{dg^{(k)}_{j}} }_{q}
\\ &\leq
\norm{ \ell^{r_{1}}_{k} MF^{(k)} }_{q_{1}}
\norm[\big]{ \ell^{r_{0}}_{k} Sg^{(k)} }_{q_{0}}.
\end{align*}
Here and later, we abbreviate $MF^{(k} := \sup_{j} \abs{F_{\tau_{k}',j}^{(k}}$.

Consider now $q<1$.
By homogeneity, we may assume
\begin{equation}
\label{eq:vv-paraprod-scaling-assumption}
\norm[\big]{ \ell^{r_{1}}_{k} MF^{(k)} }_{q_{1}}
= \norm[\big]{ \ell^{r_{0}}_{k} Sg^{(k)} }_{q_{0}}
= 1,
\end{equation}
and we have to show
\[
\norm[\big]{\ell^{r}_{k} \sup_{\tau_{k}' \leq t \leq \tau_{k}} \abs{\Pi(F^{(k)},g^{(k)})_{\tau_{k}',t}} }_{q}
\lesssim 1.
\]

We use the Davis decomposition $g=g^{\pred}+g^{\bv}$ (Lemma~\ref{lem:Davis-decomposition:Lq} with $X=\ell^{r_{0}}$).
The contribution of the bounded variation part is estimated as follows:
\begin{align*}
\MoveEqLeft
\norm{ \ell^{r}_{k} \sup_{\tau_{k}' \leq t \leq \tau_{k}} \abs{\Pi(F^{(k)},g^{(k),\bv})_{\tau_{k}',t}} }_{q}
\\ &\leq
\norm{ \ell^{r}_{k} \sum_{j} \abs{F^{(k)}_{\tau_{k}',j-1}} \cdot \abs{dg^{(k),\bv}_{j}} }_{q}
\\ &\leq
\norm{ \ell^{r_{1}}_{k} MF^{(k)} }_{q_{1}}
\norm{ \ell^{r_{0}}_{k} \Bigl( \sum_{j} \abs{dg^{(k),\bv}_{j}} \Bigr) }_{q_{0}}
\\ &\leq
\norm{ \ell^{r_{1}}_{k} MF^{(k)} }_{q_{1}}
\norm{ \sum_{j} \ell^{r_{0}}_{k} \abs{dg^{(k),\bv}_{j}} }_{q_{0}}
\\ &\lesssim
\norm{ \ell^{r_{1}}_{k} MF^{(k)} }_{q_{1}}
\norm{ \sup_{j} \ell^{r_{0}}_{k} \abs{dg^{(k)}_{j}} }_{q_{0}}
\\ &\leq
\norm{ \ell^{r_{1}}_{k} MF^{(k)} }_{q_{1}}
\norm{ \ell^{r_{0}}_{k} Sg^{(k)} }_{q_{0}},
\end{align*}
where we used \eqref{eq:davis:bv} in the penultimate step.

It remains to consider the part $g^{\pred}$ with predictable bounds for jumps.
We use the layer cake formula in the form
\[
\int f^{q}
= \int_{0}^{\infty} \P \Set{ f^{q} > \lambda } \dif \lambda
= \int_{0}^{\infty} \P \Set{ f > \lambda^{1/q} } \dif \lambda.
\]
By the layer cake formula, we have
\begin{multline}
\label{eq:vv-paraprod-layer-cake}
\norm[\big]{ \ell^{r}_{k} \sup_{\tau_{k}' \leq t \leq \tau_{k}} \abs{\Pi(F^{(k)},g^{(k),\pred})_{\tau_{k}',t}} }_{q}^{q}
\\ =
\int_{0}^{\infty} \P \Set{ \ell^{r}_{k} \sup_{\tau_{k}' \leq t \leq \tau_{k}} \abs{\Pi(F^{(k)},g^{(k),\pred})_{\tau_{k}',t}} > \lambda^{1/q} } \dif \lambda.
\end{multline}
Fix some $\lambda>0$ and define a stopping time
\begin{equation}
\label{eq:vv-paraprod:stopping}
\tau := \inf \Set[\Big]{ t \given
M(\ell^{r_{0}}_{k} d g^{(k)})_{t} > c\lambda^{1/q_{0}} \text{ or }
\ell^{r_{0}}_{k} S g^{(k),\pred}_{t} > c\lambda^{1/q_{0}} \text{ or }
\ell^{r_{1}}_{k} \sup_{0 < j \leq t} \abs{F_{\tau_{k}',j}^{(k)}} > \lambda^{1/q_{1}} }.
\end{equation}
Define stopped martingales $\tilde{g}^{(k)}_{t} := g^{(k),\pred}_{t \wedge \tau}$ and adapted processes
\[
\tilde{F}^{(k)}_{t',t} :=
F^{(k)}_{t', t \bmin \tau-1}.
\]
Then, on the set $\Set{\tau=\infty}$, we have
\[
\Pi(F^{(k)},g^{(k),\pred})_{\tau_{k}',t}
=
\Pi(\tilde{F}^{(k)},\tilde{g}^{(k)})_{\tau_{k}',t}
\quad\text{for all } k,t.
\]
Hence,
\begin{equation}
\label{eq:10}
\begin{split}
\MoveEqLeft
\Set{ \ell^{r}_{k} \sup_{\tau_{k}' \leq t \leq \tau_{k}} \abs{\Pi(F^{(k)},g^{(k),\pred})_{\tau_{k}',t}} > \lambda^{1/q} }
\\ \subset&
\Set{ \ell^{r}_{k} \sup_{\tau_{k}' \leq t \leq \tau_{k}} \abs{\Pi(\tilde{F}^{(k)},\tilde{g}^{(k)})_{\tau_{k}',t}} > \lambda^{1/q} }
\\ &\cup
\Set{ \ell^{r_{0}}_{k} S g^{(k)} > \lambda^{1/q_{0}} }
\cup \Set{ \ell^{r_{0}}_{k} S g^{(k),\pred} > \lambda^{1/q_{0}} }
\\ &\cup
\Set{ \ell^{r_{1}}_{k} MF^{(k)} > \lambda^{1/q_{1}} }
\end{split}
\end{equation}
The contributions of the latter three terms to \eqref{eq:vv-paraprod-layer-cake} are $\lesssim 1$ by \eqref{eq:vv-paraprod-scaling-assumption} and Lemma~\ref{lem:davis-dec:Spred}.
It remains to handle the first term.

By construction, we have $\ell^{r_{1}}_{k} M\tilde{F}^{(k)} \leq \lambda^{1/q_{1}}$, and due to \eqref{eq:davis-predictable} we also have $\ell^{r_{0}}_{k} S \tilde{g}^{(k)} \leq \lambda^{1/q_{0}}$, provided that the absolute constant $c$ in \eqref{eq:vv-paraprod:stopping} is small enough.
Choose an arbitrary exponent $\tilde{q}$ with $q_{0} < \tilde{q} < \infty$.
By the already known case of the Proposition with $(q_{0},q_{1})$ replaced by $(\tilde{q},\infty)$, we obtain
\begin{equation}
\label{eq:12}
\begin{split}
\MoveEqLeft
\P \Set{ \ell^{r}_{k} \sup_{\tau_{k}' \leq t \leq \tau_{k}} \abs{\Pi(\tilde{F}^{(k)},\tilde{g}^{(k)})_{\tau_{k}',t}} > \lambda^{1/q} }
\\ &\leq
\lambda^{-\tilde{q}/q} \norm{ \ell^{r}_{k} \sup_{\tau_{k}' \leq t \leq \tau_{k}} \abs{\Pi(\tilde{F}^{(k)},\tilde{g}^{(k)})_{\tau_{k}',t}} }_{\tilde{q}}^{\tilde{q}}
\\ &\lesssim_{\tilde{q}}
\lambda^{-\tilde{q}/q} \norm{\ell^{r_{1}}_{k} M\tilde{F}^{(k)}}_{\infty}^{\tilde{q}}
\norm{\ell^{r_{0}}_{k} S \tilde{g}^{(k)} }_{\tilde{q}}^{\tilde{q}}
\\ &\leq
\lambda^{-\tilde{q}/q_{0}} \norm{\ell^{r_{0}}_{k} S g^{(k),\pred} \wedge \lambda^{1/q_{0}}}_{\tilde{q}}^{\tilde{q}}.
\end{split}
\end{equation}
This estimate no longer depends on the stopping time $\tau$.
Integrating the right-hand side of \eqref{eq:12} in $\lambda$, we obtain
\begin{align*}
\int_{0}^{\infty} \lambda^{-\tilde{q}/q_{0}} \norm{\ell^{r_{0}}_{k} S g^{(k),\pred} \wedge \lambda^{1/q_{0}}}_{\tilde{q}}^{\tilde{q}}
\dif \lambda
&=
\mathbb{E} \int_{0}^{\infty}
\bigl( \lambda^{-\tilde{q}/q_{0}} (\ell^{r_{0}}_{k} S g^{(k),\pred})^{\tilde{q}} \wedge 1 \bigr)
\dif \lambda
\\ &\sim
\mathbb{E} (\ell^{r_{0}}_{k} S g^{(k),\pred})^{q_{0}}
\\ &\sim 1,
\end{align*}
where we used $\tilde{q}>q_{0}$, Lemma~\ref{lem:davis-dec:Spred} with $X=\ell^{r_{0}}$, and the assumption \eqref{eq:vv-paraprod-scaling-assumption}.
\end{proof}

\neuevorl{2021-12-21}

We will soon need a version of Proposition~\ref{prop:vv-pprod} with a supremum in both time arguments of the paraproduct.
Nice estimates of such form are only available under some structural assumptions on $F$.
In this course, we only consider $F=\delta f$.

\begin{theorem}
\label{thm:vv-pprod-delta-f}
Let $q,q_{0},q_{1},r,r_{1}$ be as in Proposition~\ref{prop:vv-pprod} with $r_{0}=2$, that is,
\[
0 < q,q_{1} \leq \infty,
1 \leq q_{0},r < \infty,
1 \leq r_{1} \leq \infty,
1/q = 1/q_{0} + 1/q_{1},
1/r = 1/2 + 1/r_{1}.
\]
Let $f$ be an adapted process, $g$ a martingale, and $\tau$ an adapted partition.
Then, we have
\begin{equation}
\label{eq:vv-pprod-delta-f}
\norm[\big]{ \ell^{r}_{k} \sup_{\tau_{k-1} \leq s \leq t \leq \tau_{k}} \abs{\Pi(\delta f,g)_{s,t}} }_{q}
\lesssim
\norm[\big]{ \ell^{r_{1}}_{k} \sup_{\tau_{k-1} \leq t < \tau_{k}} \abs{\delta f_{\tau_{k-1},t}} }_{q_{1}}
\norm{ Sg }_{q_{0}}.
\end{equation}
\end{theorem}
\begin{proof}
For any $s\leq t\leq u$, the sums \eqref{eq:paraprod} satisfy the relation
\begin{align*}
\delta \Pi(F,g)_{s,t,u}
&=
\Pi_{s,u}(F,g) - \Pi_{s,t}(F,g) - \Pi_{t,u}(F,g)
\\ &=
\sum_{s < j \leq u} F_{s,j-1} dg_{j} - \sum_{s < j \leq t} F_{s,j-1} dg_{j} - \sum_{t < j \leq u} F_{t,j-1} dg_{j}
\\ &=
\sum_{t < j \leq u} (F_{s,j-1} - F_{t,j-1} ) dg_{j}.
\end{align*}
In case of $F=\delta f$, the right-hand side becomes
\[
\sum_{t < j \leq u} (f_{t} - f_{s} ) dg_{j}
=
(f_{t} - f_{s} ) (g_{u}-g_{t}).
\]
Therefore, for any $\tau_{k}' \leq s \leq t$, we can estimate
\[
\abs{\Pi_{s,t}(\delta f,g)}
\leq
\abs{\Pi_{\tau_{k}',t}(\delta f,g)} + \abs{\Pi_{\tau_{k}',s}(\delta f,g)}
+ \abs{\delta f_{\tau_{k}',s} \delta g_{s,t}}.
\]
In the first two terms, we apply Proposition~\ref{prop:vv-pprod} with $\tau_{k}'=\tau_{k-1}$, $f^{(k)}=f$, $g^{(k)}=g$ for each $k$.
In the last term, by H\"older's inequality, we have
\[
\norm[\big]{ \ell^{r}_{k} \sup_{\tau_{k-1} \leq s < t \leq \tau_{k}} \abs{\delta f_{\tau_{k}',s} \delta g_{s,t}} }_{q}
\leq
\norm[\big]{ \ell^{r_{1}}_{k} \sup_{\tau_{k-1} \leq s < t \leq \tau_{k}} \abs{\delta f_{\tau_{k}',s}} }_{q_{1}}
\norm[\big]{ \ell^{2}_{k} \sup_{\tau_{k-1} \leq s < t \leq \tau_{k}} \abs{\delta g_{s,t}} }_{q_{0}}.
\]
In the former norm, we observe that the dependence on $t$ disappears.
In the latter norm, we use the $\ell^{2}$ valued BDG inequality (Corollary~\ref{cor:lr-BDG}) with the martingales $h^{(k)} = \startat{\tau_{k-1}} g^{\tau_{k}}$.
\end{proof}

\subsection{Stopping time construction}
\label{sec:var-stopping}
In this section, we estimate the $r$-variation of a two-parameter function by square function-like objects, like we did this for one-parameter functions in the proof of L\'epingle's inequality.

For an adapted process $(\Pi_{s,t})_{s \leq t}$, let
\[
\Pimax_{n''} := \sup_{ 0 \leq n < n' \leq n''} \abs{\Pi_{n,n'}},
\quad
\Pimax := \Pimax_{\infty}.
\]
\begin{lemma}
\label{lem:paraprod-stopping}
For any discrete time adapted process $(\Pi_{s,t})_{s<t}$, there exist adapted partitions $\tau^{(m)}_{j}$ such that, for every $0<\rho<r<\infty$, we have
\begin{multline}
\label{eq:5}
\sup_{\substack{l_{\max},\\ u_{0} < \dotsb < u_{l_{\max}}}}
\sum_{l=1}^{l_{\max}} \abs{\Pi_{u_{l-1},u_{l}}}^{r}
\leq
\frac{(\Pimax)^{r}}{1-2^{-r}}
+ 2^{\rho} \sum_{m=0}^{\infty} (2^{-m}\Pimax)^{r-\rho} \sum_{j=1}^{\infty} \Bigl( \sup_{\tau^{(m)}_{j-1} \leq t < \tau^{(m)}_{j}} \abs{\Pi_{t, \tau^{(m)}_{j}}} \Bigr)^{\rho}.
\end{multline}
\end{lemma}

\begin{proof}[Proof of Lemma~\ref{lem:paraprod-stopping}]
For $m\in\N$, define stopping times
\[
\tau^{(m)}_{0} := 0,
\]
and then, for $j \ge 0$, allowing values in $\N \cup \Set{\infty}$,
\begin{equation}
\label{eq:stopping-time}
\tau^{(m)}_{j+1} := \inf \Set[\Big]{ t > \tau^{(m)}_{j} \given \sup_{\tau^{(m)}_{j} \leq t' < t} \abs{ \Pi_{t', t} } > 2^{-m-1} \Pimax_{t} }.
\end{equation}
Fix $\omega\in\Omega$ and let $(u_{l})_{l=0}^{l_{\max}}$ be a finite strictly increasing sequence.
Consider $0<\rho<r<\infty$ and split
\begin{equation}
\label{eq:4}
\sum_{l=1}^{l_{\max}} \abs{\Pi_{u_{l-1},u_{l}}}^{r}
=
\sum_{m=0}^{\infty} \sum_{l \in L(m)} \abs{\Pi_{u_{l-1},u_{l}}}^{r},
\end{equation}
where
\begin{equation}
\label{eq:3}
L(m) := \Set[\big]{ l \in \Set{1,\dotsc,l_{\max}} \given 2^{-m-1}\Pimax_{u_{l}} < \abs{\Pi_{u_{l-1},u_{l}}} \leq 2^{-m}\Pimax_{u_{l}} }.
\end{equation}
In \eqref{eq:4}, we only omitted vanishing summands, since $|\Pi_{u_{l-1},u_l}| \leq \Pimax_{u_l}$.
Let also $L'(m) := L(m) \setminus \Set{\sup L(m)}$.
Using \eqref{eq:3}, we obtain
\begin{equation}
\label{eq:1}
\sum_{l=1}^{l_{\max}} \abs{\Pi_{u_{l-1},u_{l}}}^{r}
\leq
\sum_{m=0}^{\infty} (2^{-m}\Pimax)^{r-\rho} \sum_{l \in L'(m)} \abs{\Pi_{u_{l-1},u_{l}}}^{\rho}
+ \sum_{m=0}^{\infty} (2^{-m} \Pimax)^{r}.
\end{equation}

\begin{claim}
For every $l \in L(m)$, there exists $j$ s.t. $\tau_j^{(m)} \in (u_{l-1},u_{l}]$.
\end{claim}
\begin{proof}[Proof of the claim]
Let $j$ be maximal with $\tau^{(m)}_{j} \leq u_{l-1}$.
Since $l\in L(m)$, by definition \eqref{eq:3}, we have
\[
\abs{\Pi_{u_{l-1},u_{l}}}
>
2^{-m-1} \Pimax_{u_{l}}.
\]
By the definition of stopping times \eqref{eq:stopping-time}, we obtain $\tau^{(m)}_{j+1} \leq u_{l}$.
\end{proof}

Fix $m$.
For each $l \in L'(m)$, let $j(l)$ be the largest $j$ such that $\tau^{(m)}_{j} \in (u_{l-1}, u_{l}]$.
Then all $j(l)$ are distinct, and, since $l \neq \max L(m)$, the claim shows that $\tau^{(m)}_{j(l)+1} < \infty$.
Furthermore, by \eqref{eq:3}, the monotonicity of $t \mapsto \Pimax_{t}$, and the fact that the infimum in the definition \eqref{eq:stopping-time} of stopping times is in fact a minimum unless it is infinite, we have
\begin{equation}
\label{eq:2}
\abs{\Pi_{u_{l-1},u_{l}}}
\leq
2^{-m} \Pimax_{u_{l}}
\leq
2^{-m} \Pimax_{\tau^{(m)}_{j(l)+1}}
\leq
2 \sup_{\tau^{(m)}_{j(l)} \leq t' < \tau^{(m)}_{j(l)+1}} \abs{ \Pi_{t', \tau^{(m)}_{j(l)+1}}}.
\end{equation}
Since all $j(l)$ are distinct, this implies
\[
\sum_{l \in L'(m)} \abs{\Pi_{u_{l-1},u_{l}}}^{\rho}
\leq 2^{\rho}
\sum_{j=1}^{\infty} \sup_{\tau^{(m)}_{j-1} \leq t' < \tau^{(m)}_{j}} \abs{ \Pi_{t', \tau^{(m)}_{j}}}^{\rho}.
\]
Substituting this into \eqref{eq:1}, we conclude the proof of Lemma~\ref{lem:paraprod-stopping}.
\end{proof}

\begin{corollary}
\label{cor:paraprod-stopping}
Let $(\Pi_{s,t})_{s \leq t}$ be an adapted process with $\Pi_{t,t}=0$ for all $t$.
Then, for every $0 < \rho < r < \infty$ and $q \in (0,\infty]$, we have
\begin{equation}
\label{eq:Vr-stopping}
\norm{ V^{r} \Pi }_{L^{q}}
\lesssim
\sup_{\tau}
\norm[\Big]{ \Bigl( \sum_{j=1}^{\infty} \bigl( \sup_{\tau_{j-1} \leq t < t' \leq \tau_{j}} \abs{\Pi_{t, t'}} \bigr)^{\rho} \Bigr)^{1/\rho} }_{L^{q}},
\end{equation}
where the supremum is taken over all adapted partitions $\tau$.
\end{corollary}
\begin{proof}
By the monotone convergence theorem, we can restrict the times in the definition of $V^{r}$ to a finite set, and then apply Lemma~\ref{lem:paraprod-stopping}.

The term $\Pimax$ is of the form on the right-hand side of \eqref{eq:Vr-stopping} with $\tau_{1}=\infty$.
Therefore, the claim follows from the triangle inequality in $L^{q}$ (if $q\geq 1$), $q$-convexity of $L^{q}$ (if $q<1$), and H\"older's inequality.
\end{proof}

\begin{theorem}
\label{thm:pprod-Vr}
Let
\[
0 < q,q_{1} \leq \infty,
1 \leq q_{0} < \infty,
1 \leq r_{1} \leq \infty,
1/q = 1/q_{0} + 1/q_{1},
1/r < 1/2 + 1/r_{1}.
\]
Let $f$ be an adapted process and $g$ a martingale.
Then, we have
\begin{equation}
\label{eq:pprod-Vr}
\norm{ V^{r}_{k} \Pi(\delta f,g) }_{q}
\lesssim
\norm{ V^{r_{1}} f }_{q_{1}}
\norm{ Sg }_{q_{0}}.
\end{equation}
\end{theorem}
\begin{proof}
Combine Corollary~\ref{cor:paraprod-stopping} and Theorem~\ref{thm:vv-pprod-delta-f}.
\end{proof}

\begin{remark}
This section is mostly copied from \cite{arxiv:2008.08897}.
\end{remark}

\neuevorl{2022-01-11}

\section{It\^{o} integration}

In this section, we discuss integration with respect to a \cadlag{} martingale in continuous time, generalizing the paraproduct $\Pi$ from the previous section.
We consider stochastic processes adapted to some filtration $(\calF_{t})_{t\in \R_{\geq 0}}$ indexed by positive real times.
We begin with some convenient regularity assumptions.
These assumptions are not restrictive, in the sense that, for any martingale, one can reparametrize time and change the filtration in such a way that they are satisfied, but we will not discuss this, since our focus is on more quantitative issues.

A filtration $(\calF_{t})_{t\in \R_{\geq 0}}$ is called \emph{right-continuous} if, for every $t\geq 0$, we have
\[
\calF_{t} = \bigcap_{t' > t} \calF_{t'}.
\]
For a function $f:R_{\geq 0} \to E$ with values in a metric space $E$, the left and right limits at a point $t$ are denoted by
\[
f_{t-} := \lim_{s\to t, s<t} f_{s},
\quad
f_{t+} := \lim_{s\to t, s>t} f_{s},
\]
if they exist.
A function $f : \R_{\geq 0} \to E$ is called \emph{\cadlag{}} (for ``continue à droite, limite à gauche'', ``right continuous with left limits''; some authors use the English abbreviation ``rcll'') if $f_{t-}$ exists for every $t>0$, and $f_{t}=f_{t+}$ for every $t\geq 0$.
A stochastic process $g : \Omega \times \R_{\geq 0} \to E$ is called \cadlag{} if every path $g(\omega,\cdot)$ is a \cadlag{} function.

\begin{theorem}[Regularization, see e.g.\ {\cite[Theorem 9.28]{MR4226142}}]
Let $\calF=(\calF_{t})_{t\in\R_{\geq 0}}$ be a right-continuous filtration.
If $g$ is a martingale with respect to $\calF$, then there exists a \cadlag{} martingale $\tilde{g}$ with respect to $\calF$ such that, for every $t\geq 0$, we have
\[
g(\cdot,t) = \tilde{g}(\cdot,t)
\quad\text{a.e.}
\]
\end{theorem}

For an adapted partition $\pi$, we write
\begin{equation}
\label{eq:def:floor}
\floor{t,\pi} := \max\Set{ s\in\pi \given s \leq t},
\quad 0 \leq t < \infty.
\end{equation}
For a \cadlag{} adapted process $f=(f_{t})_{t\geq 0}$, a \cadlag{} martingale $g=(g_{t})_{t\geq 0}$, and an adapted partition $\pi$, we consider the following Riemann--Stieltjes sums:
\begin{equation}
\label{eq:hPi}
\Pi^{\pi}(f,g)_{t,t'}
:=
\sum_{t < \pi_{j} < t'} \delta f_{\floor{t,\pi},\pi_{j}} \delta g_{\pi_{j},\pi_{j+1} \bmin t'},
\quad 0 \leq t \leq t' < \infty.
\end{equation}
In each summand, the integrand $f$ is evaluated at the left endpoint of the interval $[\pi_{j},\pi_{j+1}]$ on which we consider the increment of the integrator $g$ (and also at $\floor{t,\pi}$, which is even further to the left).
This is the distinguishing feature of the It\^o integral that makes it a martingale in the $t'$ variable.

Unlike in Riemann(--Stieltjes) integration with a bounded variation integrator, evaluating $t$ at other points in general produces different results (e.g.\ Stratonovich integral, where $f$ is averaged over $\pi_{j}$ and $\pi_{j+1}$).

Most classical treatments involve sums
\[
\sum_{\pi_{j} < t'} f_{\pi_{j}} \delta g_{\pi_{j},\pi_{j+1} \bmin t'},
\]
which corerspond to fixing $t=0$, but this obscures the observation of the natural regularity of the It\^o integral that fits nicely with rough path theory.

For an adpated process $f$ and an adapted partition $\pi$, we write
\begin{equation}
\label{eq:F-discrete}
f^{(\pi)}_{t} := f_{\floor{t,\pi}}.
\end{equation}

\begin{proposition}
\label{prop:cont-pprod-bd}
Let $0 < q_{1} \leq \infty$, $1\leq q_{0} < \infty$, and $0 < r,p_{1} \leq \infty$.
Suppose
\begin{equation}
\label{eq:Vr-exponent-condition}
1/r < 1/p_{1} + 1/2,
\quad
1/q = 1/q_{0} + 1/q_{1}.
\end{equation}
Let $(f_{t})$ be a \cadlag{} adapted process and $(g_{t})$ a \cadlag{} martingale.

Then, for every adapted partition $\pi$, we have the estimate
\begin{equation}
\label{eq:YoungBDG}
\norm[\big]{ V^{r} \Pi^{\pi}(f,g) }_{L^{q}}
\lesssim
\norm{ V^{p_1} f^{(\pi)} }_{L^{q_1}} \norm{ V^{\infty}g }_{L^{q_{0}}}.
\end{equation}
\end{proposition}

\begin{proof}[Proof of Proposition~\ref{prop:cont-pprod-bd}]
Since $\Pi^{\pi}(F,g)_{t,t'}$ is \cadlag{} in both $t$ and $t'$, we have
\[
V^{r} \Pi^{\pi}(F,g) =
\lim_{n\to\infty} \sup_{l_{\max}, u_{0} < \dotsb < u_{l_{\max}}, u_{l} \in \pi^{(n)}}
\Bigl( \sum_{l=1}^{l_{\max}} \abs{ \Pi^{\pi}(F,g)_{u_{l-1},u_l} }^{r} \Bigr)^{1/r},
\]
where $\pi^{(n)} = \pi \cup 2^{-n}\N$.
By the monotone convergence theorem, it suffices to consider a fixed $\pi^{(n)}$, as long as the bound does not depend on $n$.

For any adapted partitions $\pi \subseteq \tau$, we have
\begin{equation}
\label{eq:Pi-pi'-F-pi}
\begin{split}
\Pi^{\pi}(f,g)_{t,t'}
&=
\sum_{k : t < \pi_{k} < t'} \delta f_{\floor{t,\pi},\pi_{k}} \delta g_{\pi_{k},\pi_{k+1} \bmin t'}
\\ &=
\sum_{k : t < \pi_{k} < t'} \delta f_{\floor{t,\pi},\pi_{k}}
\sum_{l : \pi_{k} \leq \tau_{l} < \pi_{k+1} \bmin t'} \delta g_{\tau_{l},\tau_{l+1} \bmin t'}
\\ &=
\sum_{k : t < \pi_{k} < t'}
\sum_{l : \pi_{k} \leq \tau_{l} < \pi_{k+1} \bmin t'}
\delta f_{\floor{t,\pi},\floor{\tau_{l},\pi}}
\delta g_{\tau_{l},\tau_{l+1} \bmin t'}
\\ &=
\sum_{l : t < \tau_{l} < t'}
\delta f^{(\pi)}_{\floor{t,\tau},\tau_{l}}
\delta g_{\tau_{l},\tau_{l+1} \bmin t'}
\\ &=
\Pi^{\tau}(f^{(\pi)},g)_{t,t'},
\end{split}
\end{equation}
where $f^{(\pi)}$ is given by \eqref{eq:F-discrete}.
Define discrete time processes $f^{(\pi)}_{\tau},g_{\tau}$ by
\[
(f^{(\pi)}_{\tau})_{j} = f^{(\pi)}_{\tau_{j}},
\quad
(g_{\tau})_{j} = g_{\tau_{j}}.
\]
Then, we have
\begin{align*}
\Pi^{\pi}(f,g)_{\tau_{j},\tau_{j'}}
&=
\Pi^{\tau}(f^{(\pi)},g)_{\tau_{j},\tau_{j'}}
\\ &=
\sum_{l : \tau_{j} < \tau_{l} < \tau_{j'}}
\delta f^{(\pi)}_{\floor{\tau_{j},\tau},\tau_{l}}
\delta g_{\tau_{l},\tau_{l+1} \bmin \tau_{j'}}
\\ &=
\sum_{l : j < l < j'}
\delta f^{(\pi)}_{\tau_{j},\tau_{l}}
\delta g_{\tau_{l},\tau_{l+1}}
\\ &=
\Pi(f^{(\pi)}_{\tau},g_{\tau})_{j,j'},
\end{align*}
where the last line is the discrete time paraproduct defined in \eqref{eq:paraprod}.
By Theorem~\ref{thm:pprod-Vr} and the BDG inequality (Corollary~\ref{cor:BDG}) for the discrete time martingale $g_{\tau}$, we obtain
\begin{multline*}
\norm{V^{r} \Pi(f^{(\pi)}_{\tau},g_{\tau}) }_{q}
\lesssim
\norm{ V^{r_{1}} f^{(\pi)}_{\tau} }_{q_{1}}
\norm{ Sg_{\tau} }_{q_{0}}
\lesssim
\norm{ V^{r_{1}} f^{(\pi)}_{\tau} }_{q_{1}}
\norm{ V^{\infty}g_{\tau} }_{q_{0}}
\leq
\norm{ V^{r_{1}} f }_{q_{1}}
\norm{ V^{\infty}g }_{q_{0}}.
\qedhere
\end{multline*}
\end{proof}

\begin{lemma}
\label{lem:F-pi-converges-to-F}
Let $(f_{t})_{t\geq 0}$ be a \cadlag{} adapted process.
Suppose that $V^{p_{1}}f \in L^{q_{1}}$ for some $p_{1},q_{1} \in (0,\infty]$.
Then, for every $\tilde{p}_{1} \in (p_{1},\infty) \cup \Set{\infty}$, we have
\[
\lim_{\pi} \norm{V^{\tilde{p}_{1}} (f-f^{(\pi)})}_{L^{q_{1}}} = 0.
\]
\end{lemma}
\begin{proof}
We have $V^{p_{1}}f^{(\pi)} \leq V^{p_{1}}f$ and, by Hölder's inequality,
\[
V^{\tilde{p}_{1}}(f-f^{(\pi)})
\leq
V^{p_{1}}(f-f^{(\pi)})^{1-\theta}
V^{\infty}(f-f^{(\pi)})^{\theta}
\]
with some $\theta \in (0,1]$, so it suffices to consider $\tilde{p}_{1}=\infty$.

Let $\epsilon > 0$ and define a sequence of stopping times recursively, starting with $\pi_{0} := 0$, by
\[
\pi_{j+1} := \inf \Set[\Big]{ t>\pi_{j} \given%
\abs{\delta f_{\pi_{j},t}} \geq \epsilon}.
\]
Since $f$ is \cadlag{}, the infimum is either $+\infty$ or a minimum, so that this is indeed a stopping time.
Also by the \cadlag{} assumption, this sequence of stopping times is strictly monotonically increasing in the sense that $\pi_{j} < \infty \implies \pi_{j} < \pi_{j+1}$.
Moreover, if $T:=\sup_{j} \pi_{j} < \infty$, then the left limit $f_{T-}$ does not exist, contradicting the \cadlag{} hypothesis.
Therefore, $\pi_{j}\to\infty$, so that $\pi$ is an adapted partition.

Then, by \eqref{eq:7}, for any adapted partition $\pi' \supseteq \pi$ and $s\leq t$, we have
\begin{align*}
\abs{ \delta f_{s,t} - \delta f^{(\pi')}_{s,t} }
&\leq
\abs{ f_{t} - f_{\floor{t,\pi'}}} + \abs{f_{s} - f_{\floor{s,\pi'}} }
\\ &\leq
2\epsilon + 2\epsilon.
\qedhere
\end{align*}
\end{proof}

\begin{theorem}[It\^o integral]
\label{thm:cont-pprod-convergence}
In the situation of Proposition~\ref{prop:cont-pprod-bd}, suppose that the right-hand side of \eqref{eq:YoungBDG-limit} is finite.
Then
\begin{equation}
\label{eq:paraprod-limit}
\Pi(f,g) := \lim_{\pi} \Pi^{\pi}(f,g)
\end{equation}
exists in $L^{q}(\Omega,V^{r})$, satisfies the bound
\begin{equation}
\label{eq:YoungBDG-limit}
\norm[\big]{ V^{r} \Pi(f,g) }_{L^{q}}
\lesssim
\norm{ V^{p_1} f }_{L^{q_1}} \norm{ V^{\infty}g }_{L^{q_{0}}},
\end{equation}
and, for any $0 \leq t \leq t' \leq t'' < \infty$, \emph{Chen's relation}
\begin{equation}
\label{eq:Chen}
\Pi(f,g)_{t,t''}
- \Pi(f,g)_{t,t'} - \Pi(f,g)_{t',t''}
= \delta f_{t,t'} \delta g_{t',t''}.
\end{equation}
\end{theorem}
The limit \eqref{eq:paraprod-limit} is called the \emph{It\^o integral} (with integrand $f$ and integrator $g$).
\begin{proof}[Proof of Theorem~\ref{thm:cont-pprod-convergence}]
By the Cauchy criterion for net convergence, the existence of the limit \eqref{eq:paraprod-limit} will follow if we can show that
\begin{equation}
\label{eq:14}
\lim_{\pi} \sup_{\tau \supseteq \pi} \norm[\big]{ V^{r} (\Pi^{\pi}(f,g)-\Pi^{\tau}(f,g)) }_{L^{q}} = 0.
\end{equation}
To this end, we use that, by \eqref{eq:Pi-pi'-F-pi}, we have
\[
\Pi^{\pi}(f,g)-\Pi^{\tau}(f,g)
=
\Pi^{\tau}(f^{(\pi)}-f^{(\tau)},g).
\]
Let $\tilde{p}_{1} \in (p_{1},\infty] \cup \Set{\infty}$ be such that $1/r < 1/\tilde{p}_{1}+1/2$.
By Proposition~\ref{prop:cont-pprod-bd} with $f$ replaced by $f^{(\pi)}-f^{(\tau)}$, we obtain
\begin{align*}
\MoveEqLeft
\norm[\big]{ V^{r} \Pi^{\tau}(f^{(\pi)}-f^{(\tau)},g) }_{L^{q}}
\\ &\lesssim
\norm{ V^{\tilde{p}_1} (f^{(\pi)}-f^{(\tau)})^{(\tau)} }_{L^{q_1}} \norm{ V^{\infty}g }_{L^{q_{0}}}
\end{align*}
This converges to $0$ by Lemma~\ref{lem:F-pi-converges-to-F}.

In order to show the Chen relation \eqref{eq:Chen}, we first show that the corresponding relation holds pointwise for the discretized paraproducts $\Pi^{\pi}$.
Indeed, by definition \eqref{eq:hPi}, for $t \leq t' \leq t''$, we have
\begin{equation}
\label{eq:delta-Pi-pi}
\Pi^{\pi}(f,g)_{t,t''}
- \Pi^{\pi}(f,g)_{t,t'} - \Pi^{\pi}(f,g)_{t',t''}
\end{equation}
\begin{multline*}
= \sum_{t < \pi_{j} < t''} \delta f_{\floor{t,\pi},\pi_{j}} \delta g_{\pi_{j},\pi_{j+1} \bmin t''}
- \sum_{t < \pi_{j} < t'} \delta f_{\floor{t,\pi},\pi_{j}} \delta g_{\pi_{j}, \pi_{j+1} \bmin t'}
\\- \sum_{t' < \pi_{j} < t''} \delta f_{\floor{t',\pi},\pi_{j}} \delta g_{\pi_{j}, \pi_{j+1} \bmin t''}
\end{multline*}
\begin{multline*}
= \sum_{t < \pi_{j} < t'} f_{\floor{t,\pi},\pi_{j}} (\delta g_{\pi_{j},\pi_{j+1} \bmin t''} - \delta g_{\pi_{j},\pi_{j+1} \bmin t'})
+
\sum_{t < \pi_{j} = t' < t''} \delta f_{\floor{t,\pi},\pi_{j}} \delta g_{\pi_{j},\pi_{j+1} \bmin t''}
\\ +
\sum_{t' < \pi_{j} < t''} (\delta f_{\floor{t,\pi},\pi_{j}}-\delta f_{\floor{t',\pi},\pi_{j}}) \delta g_{\pi_{j},\pi_{j+1} \bmin t''}
\end{multline*}
\begin{multline*}
= \sum_{t < \pi_{j} < t'} f_{\floor{t,\pi},\pi_{j}} \delta g_{\pi_{j+1} \bmin t',\pi_{j+1} \bmin t''}
+
\sum_{t < \pi_{j} = t' < t''} \delta f_{\floor{t,\pi},\pi_{j}} \delta g_{\pi_{j},\pi_{j+1} \bmin t''}
\\ +
\delta f_{\floor{t,\pi},\floor{t',\pi}} \sum_{t' < \pi_{j} < t''} \delta g_{\pi_{j},\pi_{j+1} \bmin t''}
\end{multline*}
All summands except possibly the one with $\pi_{j}<t'<\pi_{j+1}$ in the first sum vanish, and it follows that
\begin{multline*}
\dots =
\delta f_{\floor{t,\pi},\floor{t',\pi}} \Bigl(\sum_{t < \pi_{j} < t' < \pi_{j+1}} \delta g_{\pi_{j+1} \bmin t',\pi_{j+1} \bmin t''}
\\+
\sum_{t<\pi_{j} = t'<t''}\delta g_{\pi_{j},\pi_{j+1} \bmin t''}
+
\sum_{t' < \pi_{j} < t''} \delta g_{\pi_{j},\pi_{j+1} \bmin t''} \Bigr)
\end{multline*}
\[
=
\delta f_{\floor{t,\pi},\floor{t',\pi}} \Bigl(\sum_{t < \pi_{j} \leq t' < \pi_{j+1} \bmin t''} \delta g_{t',\pi_{j+1} \bmin t''}
+
\sum_{t' < \pi_{j} < t''} \delta g_{\pi_{j},\pi_{j+1} \bmin t''} \Bigr)
=
\delta f^{(\pi)}_{t,t'} \delta g_{t',t''}.
\]
By the already known conclusion \eqref{eq:paraprod-limit}, the process \eqref{eq:delta-Pi-pi} converges to the left-hand side of \eqref{eq:Chen}.
By Lemma~\ref{lem:F-pi-converges-to-F}, the last expression in the above chain of equalities converges to the right-hand side of \eqref{eq:Chen}.
\end{proof}

\subsection{Quadratic covariation}

\begin{proposition}
\label{prop:covariation-exists}
Let $q_{0},q_{1} \in [1,\infty)$, $1/q=1/q_{0}+1/q_{1}$, and $r>1$.
For \cadlag{} martingales $f,g$ with $V^{\infty}f\in L^{q_{0}}$, $V^{\infty}g \in L^{q_{1}}$, let
\[
[f,g]^{\pi}_{t,t'} :=
\sum_{\floor{t,\pi} \leq \pi_{j} < \floor{t',\pi}} \delta f_{\pi_{j}, \pi_{j+1}} \delta g_{\pi_{j}, \pi_{j+1}}.
\]
Then,
\begin{equation}
\label{eq:covariation}
[f,g]_{t,t'} := \lim_{\pi} [f,g]^{\pi}_{t,t'}
\end{equation}
exists in $L^{q}(V^{r})$.
Moreover, we have the integration by parts formula
\begin{equation}
\label{eq:int-parts}
\delta f_{t,t'} \delta g_{t,t'}
= \Pi(f,g)_{t,t'} + \Pi(g,f)_{t,t'} + [f,g]_{t,t'}.
\end{equation}
\end{proposition}

The process \eqref{eq:covariation} is called the \emph{quadratic covariation} of $f$ and $g$.
Here are a few facts about it that easily follow from the definition:
\[
\forall t\leq t'\leq t'',
\quad
[f,g]_{t,t''} = [f,g]_{t,t'} + [f,g]_{t',t''}
\quad\text{a.s.}
\]
To see this, note that the same identity holds for $[\cdot,\cdot]^{\pi}$ for any adapted partition $\pi$ that contains the times $t,t',t''$.
\[
\forall t\leq t', \quad
\E_{t} [f,f]_{t,t'} = \E_{t} \abs{\delta f_{t,t'}}^{2}.
\]
To see this, note that the same identity holds for $[\cdot,\cdot]^{\pi}$ for any adapted partition $\pi$ that contains the times $t,t'$, because, by the optinal sampling theorem, $\delta f_{\pi_{j},\pi_{j+1}}$ are martingale increments, and therefore orthogonal in $L^{2}$.

\begin{proof}[Proof of Proposition~\ref{prop:covariation-exists}]
First, we compute
\begin{align*}
\MoveEqLeft
\delta f_{\floor{t,\pi},t'} \delta g_{\floor{t,\pi},t'}
- \Pi^{\pi}(f,g)_{t,t'} - \Pi^{\pi}(g,f)_{t,t'}
\\ &=
\bigl( \sum_{\floor{t,\pi} \leq \pi_{k} < t'} \delta f_{\pi_{k}, \pi_{k+1} \bmin t'} \bigr)
\bigl( \sum_{\floor{t,\pi} \leq \pi_{j} < t'} \delta g_{\pi_{j}, \pi_{j+1} \bmin t'} \bigr)
\\ &-
\sum_{\floor{t,\pi} < \pi_{j} < t'} \delta f_{\floor{t,\pi},\pi_{j}} \delta g_{\pi_{j}, \pi_{j+1} \bmin t'}
-
\sum_{\floor{t,\pi} < \pi_{k} < t'} \delta g_{\floor{t,\pi},\pi_{k}} \delta f_{\pi_{k}, \pi_{k+1} \bmin t'}
\\ &=
\sum_{\floor{t,\pi} \leq \pi_{j} < t'} \delta f_{\pi_{j}, \pi_{j+1} \bmin t'} \delta g_{\pi_{j}, \pi_{j+1} \bmin t'}
\\ &+
\sum_{\floor{t,\pi} \leq \pi_{k} < \pi_{j} < t'} \delta f_{\pi_{k}, \pi_{k+1} \bmin t'}
\delta g_{\pi_{j}, \pi_{j+1} \bmin t'}
-
\sum_{\floor{t,\pi} < \pi_{j} < t'} \delta f_{\floor{t,\pi},\pi_{j}} \delta g_{\pi_{j}, \pi_{j+1} \bmin t'}
\\ &+
\sum_{\floor{t,\pi} \leq \pi_{j} < \pi_{k} < t'} \delta f_{\pi_{k}, \pi_{k+1} \bmin t'}
\delta g_{\pi_{j}, \pi_{j+1} \bmin t'}
-
\sum_{\floor{t,\pi} < \pi_{k} < t'} \delta g_{\floor{t,\pi},\pi_{k}} \delta f_{\pi_{k}, \pi_{k+1} \bmin t'}.
\end{align*}
Each of the last two lines vanishes identically.

In particular, replacing $t'$ by $\floor{t',\pi}$, we obtain
\begin{equation}
\label{eq:9}
\delta f_{\floor{t,\pi},\floor{t',\pi}} \delta g_{\floor{t,\pi},\floor{t',\pi}}
- \Pi^{\pi}(f,g)_{t,\floor{t',\pi}} - \Pi^{\pi}(g,f)_{t,\floor{t',\pi}}
=
\sum_{\floor{t,\pi} \leq \pi_{j} < \floor{t',\pi}} \delta f_{\pi_{j}, \pi_{j+1}} \delta g_{\pi_{j}, \pi_{j+1}}.
\end{equation}
Note that
\begin{equation}
\label{eq:8}
\Pi^{\pi}(f,g)_{t,\floor{t',\pi}}
=
\Pi^{\pi}(f,g)_{t,t'} - \sum_{\pi_{j} < t' < \pi_{j+1}} \delta f_{\floor{t,\pi},\pi_{j}} \delta g_{\pi_{j},t'},
\end{equation}
where the sum consists of either $0$ or $1$ summands.
For any inrcreasing sequence $u_{0} \leq \dotsc \leq u_{K}$, we have
In particular, with any $\alpha \in (0,r-1)$, we have
\begin{align*}
\MoveEqLeft
\sum_{k=0}^{K-1} \abs[\Big]{ \sum_{\pi_{j} < u_{k+1} < \pi_{j+1}} \delta f_{\floor{u_{k},\pi},\pi_{j}} \delta g_{\pi_{j},u_{k+1}} }^{r}
\\ &=
\sum_{k : \floor{u_{k},\pi} < \floor{u_{k+1},\pi}} \abs{ \delta f_{\floor{u_{k},\pi},\floor{u_{k+1},\pi}} \delta g_{\floor{u_{k+1},\pi},u_{k+1}} }^{r}
\\ &\leq
\Bigl( \sum_{k : \floor{u_{k},\pi} < \floor{u_{k+1},\pi}} \abs{ \delta f_{\floor{u_{k},\pi},\floor{u_{k+1},\pi}} }^{2r} \Bigr)^{1/2}
\Bigl( \sum_{k : \floor{u_{k},\pi} < \floor{u_{k+1},\pi}} \abs{ \delta g_{\floor{u_{k+1},\pi},u_{k+1}} }^{2r} \Bigr)^{1/2}
\\ &\leq
(V^{2r} f)^{r} \sup_{t} \abs{ \delta g_{\floor{t,\pi},t} }^{\alpha} (V^{2(r-\alpha)} g)^{r-\alpha}.
\end{align*}
This expression no longer depends on the sequence $(u_{k})$.
By L\'epingle's and Hölder's inequalities, this is bounded in $L^{q}$, and, taking into account Lemma~\ref{lem:F-pi-converges-to-F}, this converges to $0$ in $L^{q}$.
Hence, the second term on the right-hand side of \eqref{eq:8} converges to $0$ in $L^{q}V^{r}$.
Theorem~\ref{thm:cont-pprod-convergence} and L\'epingle's inequality for $f$ now imply
\[
\lim_{\pi} \Pi^{\pi}(f,g)_{t,\floor{t',\pi}}
=
\Pi(f,g)_{t,t'}
\quad\text{in } L^{q}V^{r}.
\]
Writing \eqref{eq:9} as
\begin{equation}
[f,g]^{\pi}_{t,t'}
=
\delta f^{(\pi)}_{t,t'} \delta g^{(\pi)}_{t,t'}
- \Pi^{\pi}(f,g)_{t,\floor{t',\pi}} - \Pi^{\pi}(g,f)_{t,\floor{t'}},
\end{equation}
we see that the first summand on the right-hand side converges to $\delta f \delta g$ in $L^{q}V^{r}$ by Lemma~\ref{lem:F-pi-converges-to-F}, and the remaining summands to It\^o integrals by the above discussion.
Hence, the left hand side converges, as was claimed in \eqref{eq:covariation}.
The identity \eqref{eq:int-parts} is the limit of the above equality.
\end{proof}

\neuevorl{2022-01-18}

\section{A sharp inequality for the square function}
The proof of the BDG inequality that we have seen in this notes was quite indirect: we started with some $L^{2}$ identities and used the Davis decomposition to lower the $L^{p}$ exponent.
In this section, we will take a look at a more direct method for proving martingale inequalities.
In this special case, it will yield an inequality with an optimal constant.

This section follows \cite{MR1859027}.

\begin{theorem}
\label{thm:sharp-S}
Let $(f_{n})_{n\in\N}$ be a real-valued martingale.
Then, for every $N\in\N$, we have
\begin{equation}
\label{eq:sharp-S-quot}
\E \Bigl( 3\abs{f_{0}} + \sum_{n=1}^{N} \frac{\abs{df_{n}}^{2}}{f^{*}_{n}} \Bigr)
\leq
\E\Bigl( 2 f^{*}_{N} + \frac{\abs{f_{N}}^{2}}{f^{*}_{N}} \Bigr).
\end{equation}
\end{theorem}

We will see later that the inequality \eqref{eq:sharp-S-quot} implies the Davis inequality for the martingale square function with the sharp constant.
In applications, one replaces \eqref{eq:sharp-S-quot} by the slightly weaker inequality
\[
\E \Bigl( 3\abs{f_{0}} + \sum_{n=1}^{N} \frac{\abs{df_{n}}^{2}}{f^{*}_{n}} \Bigr)
\leq
3 \E f^{*}_{N}.
\]
However, the preculiar form of the right-hand side of \eqref{eq:sharp-S-quot} permits to show this result \emph{by induction on $N$}.

\subsection{A Bellman function}
The inductive step in the proof of Theorem~\ref{thm:sharp-S} is usually stated as a concavity property of a special function.
Functions used in such arguments are called ``Bellman functions''; many more examples can be found in the books \cite{MR2964297,VasVol_Bellman_book}.

Thoroughout this section,
\[
\calD := \Set{ (x,t,z) \in \R\times\R_{\geq 0}\times\R_{\geq 0} \given \abs{x}\leq z }.
\]
We define $U : \calD \to \R$ by
\[
U(x,y,m) := y - \frac{\abs{x}^{2}+(\gamma-1)m^{2}}{m},
\]
where $\gamma = 3$.
The main feature of this function is the following concavity property.
\begin{proposition}
\label{prop:concavity}
For any $x,h \in \R$ and $y,m \in \R_{\geq 0}$ with $\abs{x}\leq m$, we have
\begin{equation}
\label{eq:inductive-step}
U(x+h,y+\frac{\abs{h}^{2}}{(\abs{x+h}\vee m)},\abs{x+h}\vee m)
\leq
U(x,y,m) - \frac{2(x) h}{m}.
\end{equation}
\end{proposition}

\begin{proof}[Proof of Theorem~\ref{thm:sharp-S} assuming Proposition~\ref{prop:concavity}]
Using \eqref{eq:inductive-step} with
\[
x = f_{n},
\quad
y = \tilde{S}_{n} := \gamma\abs{f_{0}}
+ \sum_{j=1}^{n} \frac{\abs{df_{j}}^{2}}{f^{*}_{j}},
\quad
m = f^{*}_{n},
\quad
h = df_{n+1},
\]
we obtain
\begin{equation}
U(f_{n+1},\tilde{S}_{n+1},f^{*}_{n+1})
\leq
U(f_{n},\tilde{S}_{n},f^{*}_{n}) - \frac{2 f_{n} df_{n+1}}{f^{*}_{n}}.
\end{equation}
By conditional independence, we have
\[
\E (\frac{2 f_{n} df_{n+1}}{f^{*}_{n}} \abs{ \calF_{n} )
=
\frac{2f_{n}}{f^{*}_{n}} \E ( df_{n+1} } \calF_{n} )
= 0.
\]
Taking expectations, we obtain
\[
\E U(f_{n+1},\tilde{S}_{n+1},f^{*}_{n+1})
\leq
\E U(f_{n},\tilde{S}_{n},f^{*}_{n}).
\]
Iterating this inequality, we obtain
\[
\E \Bigl( 3 \abs{f_{0}} + \sum_{n=1}^{N} \frac{\abs{df_{n}}^{2}}{(f^{*}_{n})} - \frac{\abs{f_{N}}^{2}}{f^{*}_{N}} - 2 f^{*}_{N} \Bigr)
=
\E U(f_{N},\tilde{S}_{N},f^{*}_{N})
\leq
\E U(f_{0},\tilde{S}_{0},f^{*}_{0})
=
0.
\qedhere
\]
\end{proof}

\begin{remark}
The above proof in fact shows the pathwise inequality
\[
3 \abs{f_{0}} + \sum_{n=1}^{N} \frac{\abs{df_{n}}^{2}}{(f^{*}_{n})}
\leq
2 f^{*}_{N} + \frac{\abs{f_{N}}^{2}}{f^{*}_{N}}
- \sum_{n=1}^{N}
\frac{2(f_{n}) df_{n+1}}{(f^{*}_{n})}.
\]
\end{remark}

\begin{proof}[Proof of Proposition~\ref{prop:concavity}]
If $\abs{x+h}\leq m$, then
\begin{align*}
\MoveEqLeft
U(x+h,y+\frac{\abs{h}^{2}}{(\abs{x+h}\vee m)},\abs{x+h}\vee m)
\\ &=
(y+\frac{\abs{h}^{2}}{m})
-\frac{\abs{x+h}^{2}+(\gamma-1)m^{2}}{m}
\\ &=
y+\frac{\abs{h}^{2}}{m}
-
\frac{\abs{x}^{2} + 2x h + \abs{h}^{2}+(\gamma-1)m^{2}}{m}
\\ &=
y
-
\frac{\abs{x}^{2} +(\gamma-1)m^{2}}{m}
-
\frac{2 x h}{m}
\\ &=
U(x,y,m) - \frac{2(x) h}{m}.
\end{align*}
If $\abs{x+h}>m$, then we need to show
\begin{equation}
(y+\frac{\abs{h}^{2}}{\abs{x+h}})
- \frac{\abs{x+h}^{2} + (\gamma-1) \abs{x+h}^{2}}{\abs{x+h}}
\leq
y - \frac{\abs{x}^{2} + (\gamma-1)m^{2}}{m}
- \frac{2x h}{m}.
\end{equation}
This is equivalent to
\begin{equation}\label{eq:1}
\frac{\abs{h}^{2}-\gamma \abs{x+h}^{2}}{\abs{x+h}}
\leq
\frac{-\abs{x}^{2}-(\gamma-1)m^{2}}{m} - \frac{2x h}{m}.
\end{equation}
The inequality \eqref{eq:1} is equivalent to
\[
\frac{\abs{h}^{2}m}{m^{2}\abs{x+h}} - \gamma \frac{\abs{x+h}}{m}
\leq
\frac{-\abs{x+h}^{2}+\abs{h}^{2}}{m^{2}} -(\gamma-1).
\]
Let $t := \abs{x+h}/m > 1$ and $\tilde{t} := \abs{h}/m$.
Note that $\abs{t-\tilde{t}} = \abs{\abs{x+h}-\abs{h}}/m \leq \abs{x}/m \leq 1$.
With this notation, \eqref{eq:1} is equivalent to
\[
\tilde{t}^{2}/t-\gamma t
\leq
-t^{2}+\tilde{t}^{2}-(\gamma-1),
\]
or
\[
\gamma
\geq
\frac{1}{t-1} \bigl( t^{2}-1 - \tilde{t}^{2}(1-1/t) \bigr)
=
(t+1) - \tilde{t}^{2}/t.
\]
Hence, it suffices to ensure
\[
\gamma \geq
\sup_{t>1, \abs{t-\tilde{t}} \leq 1} (t+1) - \tilde{t}^{2}/t.
\]
The supremum in $\tilde{t}$ is assumed for $\tilde{t} = (t-1)$, so this condition becomes
\[
\gamma \geq
\sup_{t>1} t+1 - (t-1)^{2}/t
=
\sup_{t>1} 3 - 1/t
=
3.
\]
\end{proof}

\subsection{Sharp constant in the Davis inequality for the square function}
\begin{proposition}
\label{prop:Davis-S-upper-bd}
Let $f$ be a real-valued martingale.
Then
\[
\E Sf \leq \sqrt{3} \E f^{*}.
\]
\end{proposition}
\begin{proof}
By Hölder's inequality and \eqref{eq:sharp-S-quot}, we obtain
\begin{align*}
\E Sf
&\leq
\E \bigl( (f^{*})^{1/2} \bigl( \sum_{n=1}^{N} \frac{\abs{df_{n}}^{2}}{f^{*}_{n}} \bigr)^{1/2} \bigr)
\\ &\leq
\bigl( \E f^{*} \bigr)^{1/2}
\bigl( \E \bigl( \sum_{n=1}^{N} \frac{\abs{df_{n}}^{2}}{f^{*}_{n}} \bigr) \bigr)^{1/2}
\\ &\leq
\sqrt{3} \E f^{*}.
\qedhere
\end{align*}
\end{proof}

The next result shows that $\sqrt{3}$ is the best possible constant in Proposition~\ref{prop:Davis-S-upper-bd}.
\begin{proposition}
\label{prop:Davis-S-lower-bd}
Let $\gamma\geq 0$.
Suppose that the inequality
\begin{equation}
\label{eq:Davis-S-constant}
\E Sf \leq \gamma \E f^{*}
\end{equation}
holds for all real-valued simple martingales $f$ with $f_{0}=0$.
Then $\gamma \geq \sqrt{3}$.
\end{proposition}
\begin{proof}
Let $V: \calD \to \R$ be given by
\[
V(x, t, z) := \sqrt{t}-\gamma z.
\]
Define $U : \calD \to (-\infty, \infty]$ by
\begin{equation}
\label{eq:B5}
U(x, t, z) :=
\sup \Set[\Big]{ \E V (f_{\infty}, t+S^{2}(f), f^{*} \vee z ) \given f_{0}=x },
\end{equation}
where the supremum is taken over all simple martingales (that is, martingales $f$ that take only finitely many values).
This function will play the role of the ``best'' Bellman function for the inequality \eqref{eq:Davis-S-constant}.

Substituting the constant martingale into the definition \eqref{eq:B5}, we see that, for any $(x,t,z) \in \cal D$, we have
\begin{equation}
\label{eq:B1}
V(x, t, z) \leq U(x, t, z),
\end{equation}
We claim that, for any simple measurable function $d: \Omega \to \R$ with $\E d=0$, we have
\begin{equation}
\label{eq:B3}
\E U (x+d, t+\abs{d}^{2}, \abs{x+d} \vee z) \leq U(x, t, z).
\end{equation}
\begin{proof}[Proof of \eqref{eq:B3}]
Fix $x \in \R$ with $\abs{x} \leq z$.
Let $d : \Omega \to \R$ be a simple function with $\E d=0$ and $\P(d=s_{j})=p_{j} \in(0,1]$ for $1 \leq j \leq m$, where $\sum_{j=1}^{m} p_{j}=1$.
Choose $b_{j} \in \R$ so that
\[
U\left(x+s_{j}, t+\abs{s_{j}}^{2}, \abs{x+s_{j}} \vee z\right)>b_{j} .
\]
Then, by the definition of $U$, there exists a martingale $f^{j}$ with $f^{j}_{0} = x+s_{j}$ satisfying
\[
\E V\bigl(f_{\infty}^{j}, t+\abs{s_{j}}^{2}+S^{2}\left(f^{j}\right),(f^{j})^{*} \vee z\bigr) > b_{j} .
\]
Let $f$ be a martingale with $f_{1}=x+d$ which continues with the same distribution as $f^{j}$ (rescaled by $p_{j}$ in measure) on the set $\Set{d=s_{j}}$.
Because $\abs{x} \leq z$, we have $f^{*} \vee z=(f^{j})^{*} \vee z$ on $\Set{d=s_{j}}$.
By \eqref{eq:B5}, we have
\begin{align*}
U(x, t, z)
&\geq
E V\left(f_{\infty}, t+S^{2}(f), f^{*} \vee z\right)
\\ &=
\sum_{j=1}^{m} \int_{\Set{d=s_{j}}} V\left(f_{\infty}, t+\abs{s_{j}}^{2}+\abs{d f_{2}}^{2}+\cdots, f^{*} \vee z\right) \dif\P \\
&=
\sum_{j=1}^{m} p_{j} \E V\left(f_{\infty}^{j}, t+\abs{s_{j}}^{2}+S^{2}(f^{j}), (f^{j})^{*} \vee z \right) \\
&\geq \sum_{j=1}^{m} p_{j} b_{j}.
\end{align*}
Using the freedom in the choice of $b_{j}$'s, this implies \eqref{eq:B3}.
\end{proof}

For every $\lambda>0$, we have
\[
V(x, t, z)=\lambda V(x / \lambda, t / \lambda^{2}, z / \lambda),
\quad
U(x, t, z)=\lambda U(x / \lambda, t / \lambda^{2}, z / \lambda).
\]
Define $u,v : [-1,1] \times \R_{\geq 0} \to (-\infty,\infty]$ by
\[
v(x, t) := V(x, t, 1),
\quad
u(x, t)=U(x, t, 1).
\]
Since $t \mapsto V(x, t, z)$ is nondecreasing, the same holds for $t \mapsto u(x, t)$, so left limits exist.
We claim that
\begin{equation}
\label{eq:16}
u(1,1-) \geq u(0,2-)+u(1,1-) .
\end{equation}
To see this, let $0<s<1 < r$.
Let $d$ be a random variable such that $\P(d=-1) = r/(r+1)$ and $\P(d=r)=1/(r+1)$.
Then, \eqref{eq:B3} and scaling imply that
\begin{align*}
u(1, s)
&= U(1,s,1) \\
& \geq \frac{r}{r+1} U(0, s+1, 1)+\frac{1}{1+r} U(1+r, s+r^{2}, 1+r )
\\ &=
\frac{r}{r+1} u(0, s+1)+u(1,(s+r^{2})/(1+r)^{2})
\\ &\geq
\frac{r}{r+1} u(0, s+1)+u(1,r^{2}/(1+r)^{2}),
\end{align*}
where we used monotonicity of $u$ in the second variable in the last step.
Both summands on the RHS are increasing in both variables $s,r$.
Taking $r\to\infty$ and $s\to 1$, we obtain \eqref{eq:16}.

If $f$ is a simple martingale with $f_{0}=0$, then, by the hypothesis \eqref{eq:Davis-S-constant}, we have
\[
\E V(f_{\infty},0+S^{2}f,f^{*}\vee 1)
=
\E Sf - \gamma \E (f^{*} \vee 1)
\leq
\E Sf - \gamma \E f^{*}
\leq
0.
\]
By the definition \eqref{eq:B5}, this implies that $u(0,0)=U(0,0,1) \leq 0$.

Let $d$ be a random variable with $\P(d=1)=\P(d=-1)=1/2$.
Using \eqref{eq:B3} and the fact that $u(x, t)=u(-x, t)$, we obtain
\[
0 \geq u(0,0) \geq \frac{1}{2}[u(1,1)+u(-1,1)]=u(1,1) \geq u(1,1-) \geq v(1,1-)
=1-\gamma.
\]
This implies that $u(1,1-)$ is finite.
So \eqref{eq:16} yields $u(0,2-) \leq 0$.
Let $d$ be a random variable with $\P(d=1)=\P(d=-1)=1/2$.
By \eqref{eq:B3}, for every $\epsilon>0$, we obtain
\begin{multline*}
0 \geq
u(0,2-)
\geq
u(0,2-\epsilon)
\geq
\frac{1}{2} [u(1,3-\epsilon)+u(-1,3-\epsilon)]
\\=
u(1,3-\epsilon)
\geq
v(1,3-\epsilon)
=
\sqrt{3-\epsilon}-\gamma
\end{multline*}
Therefore, $\gamma \geq \sqrt{3-\epsilon}$.
Since $\epsilon>0$ was arbitrary, this implies $\gamma \geq \sqrt{3}$.
\end{proof}


\printbibliography[heading=bibintoc]
\end{document}